\documentclass[letterpaper,12pt,reqno]{amsart}
\usepackage{amssymb}
\usepackage{cases}
\usepackage{amsmath}
\usepackage{mathrsfs}
\usepackage{eucal}
\usepackage{amsfonts}
\usepackage[all]{xy}
\usepackage{color}
\DeclareFontFamily{OT1}{pzc}{}
\DeclareFontShape{OT1}{pzc}{m}{it}{<-> [1.15] rpzcmi}{}
\DeclareMathAlphabet{\mathzc}{OT1}{pzc}{m}{it}

\def\End{\operatorname{End}\kern-.5pt}

\def\KK{{\mathzc K\kern0pt}}
\def\vv{{\mathzc v\kern.5pt}}

\def\aq{/\kern-2pt/}
\def\Hom{\operatorname{Hom}}
\def\diag{\operatorname{diag}}

\def\La{{\Lambda}}
\def\fT{{\mathfrak T}}
\def\fS{{W}}
\def\End{{\text{\rm End}}}
\def\Hom{{\text{\rm Hom}}}
\def\ro{{\text{\rm ro}}}
\def\co{{\text{\rm co}}}

\def\bsj{{\boldsymbol j}}

\def\up{{\boldsymbol{\upsilon}}}
\def\ups{{\upsilon}}
\def\sH{{\mathcal H}}
\def\sS{{\mathcal S}}
\def\sZ{{\mathcal Z}}
\def\sD{{\mathcal D}}

\def\la{{\lambda}}

\def\bsq{{\boldsymbol q}}
\def\dq{{\dot{\boldsymbol q}}}
\def\ddq{{\ddot{\boldsymbol q}}}

\newtheorem{theorem}{Theorem}[section]
\newtheorem{lemma}[theorem]{Lemma}
\newtheorem{proposition}[theorem]{Proposition}
\newtheorem{corollary}[theorem]{Corollary}

\theoremstyle{definition}

\newtheorem{remark}[theorem]{Remark}
\newtheorem{remarks}[theorem]{Remarks}

\numberwithin{equation}{theorem}
\def\xy{{{\text{\tiny $[$}}xy{\text{\tiny$]$}}}}
\def\yx{{{\text{\tiny $[$}}yx{\text{\tiny$]$}}}}
\def\sck{{{\textsc k}}}

\def\End{\operatorname{End}\kern-.5pt}

\def\KK{{\mathzc K\kern0pt}}
\def\vv{{\mathzc v\kern.5pt}}

\def\aq{/\kern-2pt/}
\def\Hom{\operatorname{Hom}}
\def\diag{\operatorname{diag}}

\def\La{{\Lambda}}
\def\End{{\text{\rm End}}}
\def\Hom{{\text{\rm Hom}}}
\def\row{{\text{\rm row}}}

\def\bsj{{\boldsymbol j}}

\def\bsq{{\boldsymbol q}}
\def\bse{{\boldsymbol e}}
\def\up{{\boldsymbol{\upsilon}}}
\def\sH{{\mathcal H}}
\def\sS{{\mathcal S}}
\def\sZ{{\mathcal Z}}
\def\sD{{\mathcal D}}

\def\la{{\lambda}}

\def\corner{{\boldsymbol\llcorner}}

\newcommand{\ol}{\overline}
\def\fks{{\mathfrak{s}}}
\def\fku{{\mathfrak{u}}}
\def\sck{{\textsc{k}}}
\def\sce{{\textsc{e}}}
\def\scf{{\textsc{f}}}

\begin{document}

\baselineskip16pt
\def\hei{\relax}

 \title[$q$-Schur superalgebras]{Multiplication formulas and semisimplicity for $q$-Schur superalgebras}
\author{Jie Du, Haixia Gu$^\dagger$ and Zhongguo Zhou}
\address{J.D., School of Mathematics and Statistics,
University of New South Wales, Sydney NSW 2052, Australia}
\email{j.du@unsw.edu.au}
\address{H.G., School of Science, Huzhou University, Huzhou, China}
\email{ghx@zjhu.edu.cn}
\address{Z.Z., College of Science, Hohai University, Nanjing, China}
\email{zhgzhou@hhu.edu.cn}
\keywords{symmetric group, double coset, $q$-Schur superalgebra, infinitesimal $q$-Schur superalgebra}

\date{\today}

\subjclass[2010]{20B30, 20G43, 17A70, 20C08}
\thanks{$^\dagger$Corresponding author.}
\thanks{
The work was supported by a 2017 UNSW Science Goldstar Grant and the Natural Science Foundation of China (\#11501197,  \#11671234). The third author would like to thank UNSW for its hospitality during his a year visit and thank the Jiangsu Provincial Department of Education for financial support.}


\begin{abstract}
We investigate products of certain double cosets for the symmetric group and use the findings to derive some multiplication formulas for the $q$-Schur superalgebras. This gives a combinatorialisation of the relative norm approach developed in \cite{DG}. We then give several applications of the multiplication formulas, including the matrix representation of the regular representation and  a semisimplicity criterion for $q$-Schur superalgebras. We also construct infinitesimal and little $q$-Schur superalgebras directly from the multiplication formulas and develop their semisimplicity criteria. 
\end{abstract}

 \maketitle

\tableofcontents

\section{Introduction}

The beautiful Beilinson--Lusztig--MacPherson construction \cite{BLM} of quantum $\mathfrak{gl}_n$ has been generalised to the quantum affine $\mathfrak{gl}_n$ \cite{DDF, DF}, to the quantum super $\mathfrak{gl}_{m|n}$ \cite{DG}, and partially to the other classical types  \cite{BKLW,FL} and affine type $C$ \cite{FLLLW}, in which
 certain coideal subalgebras of quantum $\mathfrak{gl}_n$ (or affine ${\mathfrak{gl}}_n$) are used to form various quantum symmetric pairs associated with Hecke algebras of type $B/C/D$ or affine type $C$. A key step of these works is the establishment of certain multiplication formulas in the relevant $q$-Schur algebras or Hecke endomorphism algebras. These formulas were originally derived by geometric methods. When the geometric approach is not available in the super case, a super version of the Curtis--Scott relative norm basis \cite{J, DU}, including a detailed analysis of the explicit action on the tensor space, is used in deriving such formulas; see  \cite{DG,DGW,DGW2}. However, it is natural to expect the existence of a direct Hecke algebra method involving only the combinatorics of symmetric groups. 

 In this paper, we will develop such a method. The multiplication formulas require to compute certain structure constants associated with the double coset basis, a basis defined by 
the double cosets of a symmetric group.  Since a double coset can be described by a certain matrix with non-negative integer entries, our first step is to find formulas,  in terms of the matrix entries, of decomposing products of certain double cosets into disjoint unions of double cosets. We then use the findings to derive the multiplication formulas in $q$-Schur superalgebras; see Theorem \ref{KMF} and Corollary \ref{KMFcor}. This method simplify the calculation in \cite[\S\S2-3]{DG} using relative norms.
 
The multiplication formulas result in several applications. The first one is the matrix representation of the regular representations over any commutative ring $R$; see Theorem \ref{KeyMF}. When the ground ring $R$ is a field, we establish a criterion for the semisimplicity of $q$-Schur superalgebras (see Theorem \ref{thmqss}), generalising a quantum result of Erdmann and Nakano to the super case and a classical super result of Marko and Zubkov \cite{mz2} (cf. \cite{DN, EN}) to the quantum case.  Finally, we introduce the infinitesimal and little $q$-Schur superalgebras directly from the multiplication formulas (Theorem \ref{6.1}, Corollary \ref{little}). We also determine semisimple infinitesimal $q$-Schur superalgebras and semisimple little $q$-Schur superalgebras (Theorem \ref{thmiqs}).

 It should be interesting to point out that, unlike the traditional methods used in \cite{DNP, DFW}, our definitions do not involve quantum enveloping algberas or quantum coordinate algebras and the semisimplicity proof is also independent of the representation theory of these ambient quantum groups or algebras. 
We expect that this combinatorial approach will give further applications to various $q$-Schur superalgebras of other types in the near future.

\vspace{.2cm}
\noindent
{\bf Acknowledgement.} We thank the referee for several helpful comments.

\section{$q$-Schur superalgebras}

Let $W={\mathfrak S}_{\{1,2,\ldots,r\}}$ be the symmetric group on $r$ letters and let $S=\{s_k\mid 1\leq k<r\}$ be the set of basic transpositions $s_k=(k,k+1)$. Denote the length function with respect to $S$ by
$\ell:W\to\mathbb{N}$.

Let $R$ be a commutative ring with 1 and let $q\in R^\times$.
The Hecke algebra $\mathcal{H}_R=\sH_R(\fS)$ is a free
$R$-module with basis $\{T_w\mid w\in W\}$ and the
multiplication defined by the rules: for $s\in S$,
\begin{equation}
T_wT_s=\left\{\begin{aligned} &T_{ws},
&\mbox{if } \ell(ws)>\ell(w);\\
&(q-1)T_w+q T_{ws}, &\mbox{otherwise}.
\end{aligned}
\right.
\end{equation}
The Hecke algebra over $R=\sZ:=\mathbb Z[\up,\up^{-1}]$ and $q=\up^2$ is simply denoted by $\sH$.

Let $W_\la$ denote the parabolic subgroup
of $W$ associated with $\la=(\lambda_1,\lambda_2,\cdots,\lambda_N)\in\La(N,r)$
where $\La(N,r)=\{\la\in{\mathbb N}^N\mid |\la|:=\sum_i\la_i=r\}$. Then $W_\la$ consists of permutations that leave
invariant the following sets of integers
$${\mathbb N}_1^\la=\{1,2,\cdots,\lambda_1\},{\mathbb N}_2^\la=\{\lambda_1+1,\lambda_1+2,\cdots,\lambda_1+\lambda_2\},\cdots.$$

Let $\sD_\la:=\mathcal{D}_{W_\la}$ be
the set of all shortest coset representatives of the right
cosets of $W_\la$ in $W$. Let
$\mathcal{D}_{\lambda\mu}=\mathcal{D}_\lambda\cap\mathcal{D}^{-1}_{\mu}$ be the set of the shortest
$W_\lambda$-$W_\mu$ double coset representatives. 

For $\la,\mu\in\La(N,r)$ and
$d\in\mathcal{D}_{\la\mu}$, the subgroup $\fS_\la^d\cap
\fS_\mu=d^{-1}\fS_\la d\cap \fS_\mu$ is a parabolic subgroup associated
with a composition which is denoted by $\la d\cap\mu$. In other
words, we define
\begin{equation}\label{ladmu}
\fS_{\la d\cap\mu}=\fS_\la^d\cap \fS_\mu.
\end{equation}
The composition $\la d\cap\mu$ can be easily described in terms of the following $N\times N$-matrix $A=(a_{i,j})$ with $a_{i,j}=|{\mathbb N}^\la_i\cap d({\mathbb N}^\mu_j)|$: if $\nu^{(j)}=(a_{1,j},a_{2,j},\ldots,a_{N,j})$ denotes the $j$th column of $A$, then
\begin{equation}\label{ladmu}
 \la d\cap\mu=(\nu^{(1)},\nu^{(2)},\ldots,\nu^{(N)}).
  \end{equation}
Putting $\jmath(\la,d,\mu)=\big(|{\mathbb N}^\la_i\cap d({\mathbb N}^\mu_j)|\big)_{i,j}$, we obtain a bijection
\begin{equation}\label{jmath}
\jmath:\{(\la,d,\mu)\mid \la,\mu\in\La(N,r),d\in\sD_{\la\mu}\}\longrightarrow M(N,r),
\end{equation}
where $M(N,r)$ is the set of all $N\times N$ matrices $A=(a_{i,j})$ over $\mathbb N$ whose entries sum to $r$, i.e., $|A|:=\sum_{i,j}a_{i,j} =r$.

For $A\in M(N,r)$, if $\jmath^{-1}(A)=(\la,d,\mu)$, then $\la,\mu\in\La(N,r)$ and
\begin{equation}\label{ro co}
\la=\ro(A):=(\sum_{j=1}^Na_{1,j},\ldots,\sum_{j=1}^Na_{N,j})\,\text{ and }\,\mu=\co(A):=(\sum_{i=1}^Na_{i,1},\ldots,\sum_{i=1}^Na_{i,N}).
\end{equation}

For the definition of $q$-Schur superalgebra, we fix two nonnegative integers $m,n$ and assume $R$ has characteristic $\neq2$.  We also need the {\it parity function}
\begin{equation}\label{parity}
\widehat{h}=\begin{cases}0,&\text{if }1\leq h\leq m;\\1,&\text{if }m+1\leq h\leq m+n.
\end{cases}\end{equation}
A composition $\la$ of $m+n$ parts will be written
$$\la=(\lambda^{(0)}|\lambda^{(1)})=(\lambda^{(0)}_1,\lambda^{(0)}_2,\cdots,\lambda^{(0)}_m|\lambda^{(1)}_1,
\lambda^{(1)}_2,\cdots,\lambda^{(1)}_n)$$ to indicate the``even'' and ``odd'' parts of $\la$. Let
\begin{equation*}
\Lambda(m|n,r):=\La(m+n,r)
=\bigcup_{r_1+r_2=r}(\La(m,r_1)\times\La(n,r_2)).
\end{equation*}

For $\lambda=(\lambda^{(0)}\mid\lambda^{(1)})\in\Lambda(m|n,r)$, we also write
\begin{equation}\label{notation}
\fS_\la=
\fS_{\la^{(0)}}\fS_{\la^{(1)}}\cong \fS_{\la^{(0)}}\times \fS_{\la^{(1)}},
\end{equation}
 where
$\fS_{\lambda^{(0)}}\leq{\mathfrak S}_{\{1,2,\ldots,|\lambda^{(0)}|\}}$
and
$\fS_{\lambda^{(1)}}\leq{\mathfrak S}_{\{|\lambda^{(0)}|+1,\ldots,r\}}$ are the even and odd parts of $\fS_\la$, respectively.

Denote the Hecke algebra associated with the parabolic subgroup $W_\la$ by $\sH_\la$, which is spanned by $T_w,w\in W_\la$. The elements in $\sH_{\la}$
\begin{equation}
\xy_\la:=x_{\la^{(0)}}y_{\la^{(1)}},\;\yx_\la:=y_{\la^{(0)}}x_{\la^{(1)}},
\end{equation}
where, for $i=0,1$,
$$x_{\lambda^{(i)}}=\sum_{w\in \fS_{\lambda^{(i)}}}T_w,\qquad y_{\lambda^{(i)}}=\sum_{w\in
\fS_{\lambda^{(i)}}}(-q)^{-\ell(w)}T_w$$ 
generate $\sH_\la$-modules $R\xy_\la$, $R\yx_\la$.
Define the ``tensor space'' (cf. \cite[(8.3.4)]{DR})
\begin{equation}\label{Tmnr}
\fT_R(m|n,r)=\bigoplus_{\lambda\in\Lambda(m|
n,r)}\xy_\la\mathcal{H}_{R}.
\end{equation}
By the definition in \cite{DR}, the endomorphism algebra
$$\sS_R(m|n,r )=\End_{\mathcal{H}_R}(\fT_R(m|
n,r ))$$
is called a $q$-{\it Schur superalgebra} whose $\mathbb Z_2$-graded structure is given by
$$\sS_R(m|n,r )_i=\bigoplus_{\la,\mu\in\La(m|n,r)\atop |\la^{(1)}|+|\mu^{(1)}|\equiv i(\text{mod}2)}\Hom_{\sH_R}(\xy_\la\sH_R,\xy_\mu\sH_R)\qquad(i=0,1).$$ We will use the notation $\sS(m|n,r)$ to denote the $\up^2$-Schur algebra over $\sZ$.

We now describe a characteristic-free basis for $\sS_R(m|n,r )$.

 For
$\la,\mu\in\Lambda(m|n,r)$, let
\begin{equation}\label{Dcirc}
\mathcal{D}^\circ_{\la\mu}=\{d\in\mathcal{D}_{\la\mu}\mid
\fS^d_{\la^{(0)}}\cap \fS_{\mu^{(1)}}=1,\fS^d_{\la^{(1)}}\cap
\fS_{\mu^{(0)}}=1\}.
\end{equation}
This set is the super version of the usual $\sD_{\la\mu}$.
We need the following subsets of the $(m+n)\times(m+n)$ matrix ring $M_{m+n}(\mathbb N)$ over $\mathbb N$:
\begin{equation}\label{M(m|n)}
\aligned
M(m|n,r)&=\{\jmath(\la,d,\mu)\mid\la,\mu\in\La(m|n,r),d\in\sD_{\la\mu}^\circ\},\\
M(m|n)&=\bigcup_{r\geq0}M(m|n,r)\subseteq M_{m+n}(\mathbb N).\endaligned
\end{equation}

Following \cite[(5.3.2)]{DR}, define, for $\lambda,\mu\in\Lambda(m|n,r)$ and
$d\in\mathcal{D}^\circ_{\lambda\mu}$,
\begin{equation}\label{double coset}
T_{\fS_\lambda d \fS_\mu}:=\xy_\lambda T_dT_{\sD_{\nu}\cap W_\mu}=T_{\sD_{\nu'}\cap W_\la}T_d\xy_\mu,
\end{equation}
where $\nu=\la d\cap\mu$, $\nu'=\mu{d^{-1}}\cap\la$, and 
$T_D=\sum_{w_0\in D_0,w_1\in D_1}T_{w_0}(-q)^{-\ell(w_1)}T_{w_1}$ for any $D\subseteq W_\eta$ ($\eta=\la$ or $\mu$) with $D_i=D\cap W_{\eta^{(i)}}$ (cf. \cite[(5.3.2)]{DR}).

The element $T_{\fS_\lambda d \fS_\mu}$ is used to define an $\mathcal{H}_R$-module homomorphism $\phi_{\la\mu}^d$ on $\fT_R(m|n,r)$:
$$\phi_{\la\mu}^d(\xy_\alpha h)=\delta_{\mu,\alpha}T_{\fS_\lambda d
\fS_\mu}h, \forall \alpha\in\Lambda(m|
n,r),h\in\mathcal{H}.$$


The first assertion of the following result is given in {\cite[5.8]{DR}}, while the  last assertion for the nonquantum case was observed in \cite[\S3.1]{HKN}. Write $\phi_A:=\phi^d_{\lambda\mu}$ if $A=\jmath(\la,d,\mu)$.
\begin{lemma}\label{DR5.8} 
The set $\{\phi_A\mid
A\in M(m|n,r)\}$
forms an $R$-basis for $\sS_R(m|n,r )$. Hence, $\sS_R(m|n,r )\cong \sS(m|n,r )\otimes_{\sZ} R$. Moreover, there is an $R$-algebra isomorphism 
$$\sS_R(m|n,r)\cong\sS_R(n|m,r).$$
\end{lemma}
\begin{proof} We only need to prove the last assertion.
The Hecke algebra $\sH_R$ admits an $R$-algebra involutory automorphism $\varphi$ sending $T_s$ to $-qT_s^{-1}=(q-1)-T_s$ for all $s\in S$. Since $\varphi(x_\la)=q^{\ell(w_{0,\la})}y_\la$, where $w_{0,\la}$ is the longest element in $W_\la$ (see, e.g., \cite[(7.6.2)]{DDPW}), we have $\varphi(\xy_\la)=\varphi(x_{\la^{(0)}}y_{\la^{(1)}})=q^{\ell(w_{0,\la^{(0)}})-\ell(w_{0,\la^{(1)}})}y_{\la^{(0)}}x_{\la^{(1)}}$. If we denote by $(\xy_\la\sH_R)^\varphi$ the module obtained by twisting the action on $\xy_\la\sH_R$ by $\varphi$, i.e., $(\xy_\la h)*h'=(\xy_\la h)\varphi(h')$ for all $h,h'\in\sH_R$, then the map
$$\Phi_\la:(\xy_\la\sH_R)^\varphi\rightarrow \yx_\la\sH_R,\xy_\la h\mapsto\varphi(\xy_\la h)$$
is an $\sH_R$ module isomorphism. These $\Phi_\la$ induce an $\sH_R$ module isomorphism
$\Phi: \fT_R(m|n,r)^\varphi\longrightarrow\fT_R(n|m,r).$
Now the required isomorphism follows.
\end{proof}

\section{Decomposing products of double cosets}

Throughout the section, let $W$ be the symmetric group and let $n,r$ be positive integers. We also fix the following notation in this section:
\begin{equation}\label{notn1}
\left\{\begin{aligned}
M&=(m_{ij})\in M(n,r)\text{ with }\jmath^{-1}(M)=(\la,d,\mu),\; d_M:=d,\\
\nu_M&:=\la d\cap\mu=(m_{1,1},m_{2,1},\cdots,m_{n,1},\cdots,m_{1,n},m_{2,n},\cdots,m_{n,n}),\\
\sigma_{i,j}&=\sum_{k=1}^{j-1}\sum_{h=1}^nm_{h,k}+\sum_{k\leq i,l\geq j}m_{k,l},\\
M^+_{h,k}&=M+E_{h,k}-E_{h+1,k}, \text{ if }m_{h+1,k}\geq1,\\
M^-_{h,k}&=M-E_{h,k}+E_{h+1,k},\text{ if }m_{h,k}\geq1.
\end{aligned}\right.
\end{equation}
Moreover, to any sequence $(a_1,a_2,\ldots,a_n)$, we associate its partial sum sequence $(\widetilde a_1,\widetilde a_2,\ldots,\widetilde a_n)$ with $\widetilde a_i=a_1+\cdots+a_i$. Thus, $\widetilde\la_i=\la_1+\cdots+\la_i$ and $\widetilde m_{i,j}$ is the partial sum at the $(i,j)$-position of $\nu_M$. 
We also note that $\sigma_{i,j}=\widetilde \mu_{j-1}+m_{i,j}^\corner,$ where $m_{i,j}^\corner=\sum_{k\leq i,l\geq j}a_{k,l}$.
In particular, $\sigma_{i,1}=m_{i,1}^\corner=\widetilde{\lambda}_{i}.$

The following result will be proved at the end of the section.
\begin{theorem}\label{prod coset}
Maintain the notation in \eqref{notn1} with $\lambda=(\lambda_1,\cdots,\lambda_n)$ and, for $1\leq h\leq n$, let
$\lambda^{[h^\pm]}:=\la\pm\bse_h\mp\bse_{h+1}=\ro(M^\pm_{h,k})$, where $\bse_i=(\delta_{1,i},\ldots,\delta_{n,i})$. Then
$$\aligned
(\fS_{\lambda^{[h^+]}} 1\fS_\lambda)(\fS_\lambda d_M\fS_\mu)&=\bigcup_{k\atop m_{h+1,k}\geq1}\fS_{\lambda^{[h^+]}} d_{M^+_{h,k}}\fS_\mu,\\
(\fS_{\lambda^{[h^-]}} 1\fS_\lambda)(\fS_\lambda d_M\fS_\mu)&=\bigcup_{k\atop m_{h,k}\geq1}\fS_{\lambda^{[h^-]}} d_{M^-_{h,k}}\fS_\mu.\\
\endaligned
$$
\end{theorem}
We first describe some standard reduced expression for $d_M$.

If $m_{i,j}=0$, or $m_{i,j}>0$ but $\sigma_{i-1,j}=\widetilde{m}_{i-1,j}$
(i.e., $m_{i-1,j+1}^\llcorner=0$), set $w_{i,j}=1$; if $m_{i,j}>0$ and $\sigma_{i-1,j}>\widetilde{m}_{i-1,j}$, let
\begin{equation}\label{wij}
\begin{aligned}
w_{i,j}=\,&(s_{\sigma_{i-1,j}}s_{\sigma_{i-1,j}-1}\cdots s_{\widetilde{m}_{i-1,j}+1})\\
&(s_{\sigma_{i-1,j}+1}s_{\sigma_{i-1,j}}\cdots s_{\widetilde{m}_{i-1,j}+2})\cdot\cdots\cdot\\
&(s_{\sigma_{i-1,j}+m_{i,j}-1}s_{\sigma_{i-1,j}+m_{i,j}-2}\cdots s_{\widetilde{m}_{i,j}})
\end{aligned}
\end{equation}
and $w^+_{i,j}=s_{\sigma_{i-1,j}+1}s_{\sigma_{i-1,j}+2}\cdots s_{\sigma_{i-1,j}+m_{i,j}}w_{i,j}$ (and $w^+_{i,j}=1$ if $m_{i,j}=0$).
Note that we may rewrite $w^+_{i,j}$ as
\begin{equation}\label{w+ij}
\begin{aligned}
w^+_{i,j}=\,&s_{\sigma_{i-1,j}+1}(s_{\sigma_{i-1,j}}s_{\sigma_{i-1,j}-1}\cdots s_{\widetilde{m}_{i-1,j}+1})\\
&s_{\sigma_{i-1,j}+2}(s_{\sigma_{i-1,j}+1}s_{\sigma_{i-1,j}}\cdots s_{\widetilde{m}_{i-1,j}+2})\cdot\cdots\cdot\\
& s_{\sigma_{i-1,j}+m_{i,j}}(s_{\sigma_{i-1,j}+m_{i,j}-1}s_{\sigma_{i-1,j}+m_{i,j}-2}\cdots s_{\widetilde{m}_{i,j}}).
\end{aligned}
\end{equation}
For example, if $M=\left(\begin{smallmatrix}1&3&2\\2&1&1\\1&0&2\end{smallmatrix}\right)$
then $(\sigma_{ij})=\left(\begin{smallmatrix}6&9&10\\10&11&11\\13&13&13\end{smallmatrix}\right)$, $(\widetilde m_{ij})=\left(\begin{smallmatrix}1&7&10\\3&8&11\\4&8&13\end{smallmatrix}\right)$, and $w_{2,1}=(s_6s_5\cdots s_2)(s_7s_6\cdots s_3)=\left(\begin{smallmatrix}2&3&4&5&6&7&8\\7&8&2&3&4&5&6\end{smallmatrix}\right)$, $w_{3,1}=s_{10}s_9\cdots s_4=\left(\begin{smallmatrix}4&5&6&7&8&9&10&11\\11&4&5&6&7&8&9&10\end{smallmatrix}\right)$, and $w_{2,2}=s_9s_8=\left(\begin{smallmatrix}8&9&10\\10&8&9\end{smallmatrix}\right)$, $w_{3,2}=1$, then
$w_{2,1}w_{3,1}w_{2,2}w_{3,2}=\left(\begin{smallmatrix}1&2&3&4&5&6&7&8&9&10&11&12&13\\1&7&8&11&2&3&4&9&5&6&10&12&13\end{smallmatrix}\right),$ which is $d_M$.

\begin{lemma}[{\cite[Algorithm 2.1]{DG}}]\label{d_M}  Let $M$, $d_M$ and $M^+_{h,k}$ be given as in \eqref{notn1}. Then a reduced expression of $d_M$ is of the form
$$d_M=(w_{2,1}w_{3,1}\cdots w_{n,1})(w_{2,2}w_{3,2}\cdots w_{n,2})\cdots(w_{2,n-1}w_{3,n-1}\cdots w_{n,n-1}).
$$
If $m_{h+1,k}\geq1$, then
\begin{equation*}
\begin{aligned}
d_{M^+_{h,k}}&=(w'_{2,1}w'_{3,1}\cdots w'_{n,1})(w'_{2,2}w'_{3,2}\cdots w'_{n,2})\cdots(w'_{2,n-1}w'_{3,n-1}\cdots w'_{n,n-1}),
\end{aligned}
\end{equation*}
where, for almost all $i,j$,  $w'_{ij}=w_{ij}$, except $w'_{h+1,j}=w^+_{h+1,j}$ for $j<k$ and
\begin{equation}\label{bulcir}
\begin{aligned}
w'_{h,k}=w^\bullet_{h,k}:=\,&w_{h,k}(s_{\sigma_{h-1,k}+m_{h,k}}s_{\sigma_{h-1,k}+m_{h,k}-1}\cdots s_{\widetilde{m}_{h,k}+1}),\\
w'_{h+1,k}=w^\circ_{h+1,k}:=\,&(s_{\sigma_{h,k}+1}s_{\sigma_{h,k}}\cdots s_{\widetilde{m}_{h,k}+2})(s_{\sigma_{h,k}+2}s_{\sigma_{h,k}+1}\cdots s_{\widetilde{m}_{h,k}+3})\cdots\\
&(s_{\sigma_{h,k}+m_{h+1,k}-1}s_{\sigma_{h,k}+m_{h+1,k}-2}\cdots s_{\widetilde{m}_{h+1,k}}).
\end{aligned}
\end{equation}
In particular, $\ell(d_{M^+_{h,k}})=\ell(d_M)+\sum_{j<k}m_{h+1,j}-\sum_{j>k}m_{h,j}.$
\end{lemma}

\begin{remark}\label{dmw}
(1) We display the factors $w_{i,j}$ of $d_M$ through a matrix notation:
\begin{equation}\label{d_Mm}
d_M={\left(\begin{array}{cccc}
w_{2,1}&w_{2,2}&\cdots&w_{2,n-1}\\
w_{3,1}&w_{3,2}&\cdots&w_{3,n-1}\\
\vdots&\vdots&\cdots&\vdots\\
w_{n,1}&w_{n,2}&\cdots&w_{n,n-1}
\end{array}\right),}
\end{equation}
where $d_M$ is simply a product of the entries down column 1, then down column 2, and so on.
Note that $w_{i,j}=1$ whenever $m_{i,j}=0$ or $m^\llcorner_{i-1,j+1}=0$.


(2) Note that a product of the form $s_{h-1}s_{h-2}\cdots s_{k}$ for $h>k$ is in fact the cycle permutation $h\to h-1\to\cdots\to  k+1\to k\to h$.
 Thus, each $w_{i,j}$ is a product of cycle permutations. 
Note also that the largest number permuted (or moved) by the partial column product $w_{2,j}w_{3,j}\cdots w_{h,j}$ is $\sigma_{h-1,j}+m_{h,j}$.
\end{remark}

\begin{lemma}\label{reflection} \begin{itemize}
\item[(1)] For any non-negative integers $k,i,h$ with $0<k\leq i<h<r$,
$$s_i(s_{h}s_{h-1}\cdots s_{k})=(s_{h}s_{h-1}\cdots s_{k})s_{i+1}.$$
Hence, for $0<k\leq i< h_1<h_2<\cdots<h_l<r$,
$$\begin{aligned}
s_i&(s_{h_1}s_{h_1-1}\cdots s_{k})(s_{h_2}s_{h_2-1}\cdots s_{k+1})\cdots (s_{h_l}s_{h_l-1}\cdots s_{k+l-1})\\
=&(s_{h_1}s_{h_1-1}\cdots s_{k})(s_{h_2}s_{h_2-1}\cdots s_{k+1})\cdots (s_{h_l}s_{h_l-1}\cdots s_{k+l-1})s_{i+l}.
\end{aligned}$$

\item[(2)] With the notation given in \eqref{notn1} and \eqref{wij}, if $\sigma_{h-1,j}+m_{h,j}<l<\sigma_{h,j}$ and $l\geq\widetilde m_{h,j}+1$, then
$$s_l(w_{2,j}w_{3,j}\cdots w_{n,j})=(w_{2,j}w_{3,j}\cdots w_{n,j})s_{l+\sum_{i=h+1}^nm_{i,j}}.$$
\item[(3)] For any $1<k\leq n$, if $0<x\leq m_{h,k}$ and assume $\sum_{j=1}^{k-1}m_{h,j}+x<\la_h$, then
$$\begin{aligned}
s_{\sigma_{h-1,1}+\sum_{j=1}^{k-1}m_{h,j}+x} &(w_{2,1}\cdots w_{n,1})\cdots (w_{2,k-1}\cdots w_{n,k-1})\\
=\,&(w_{2,1}\cdots w_{n,1})\cdots (w_{2,k-1}\cdots w_{n,k-1})s_{\sigma_{h-1,k}+x}.\end{aligned}$$
\end{itemize}
\end{lemma}
\begin{proof}
The proof for the first two assertions is straightforward. We now prove (3).

Consider the product $\prod_t$ of first $t$ columns of $d_M$:
$$\Pi_t=(w_{2,1}\cdots w_{h-1,1}w_{h,1}w_{h+1,1}\cdots w_{n,1})\cdots\cdots(w_{2,t}\cdots w_{h-1,t}w_{h,t}w_{h+1,t}\cdots w_{n,t}).$$
We claim for all $t<k$ that 
\begin{equation}\label{*}
s_{\sigma_{h-1,1}+\sum_{j=1}^{k-1}m_{h,j}+x}\cdot\Pi_t=\Pi_t\cdot s_{\sigma_{h-1,t+1}+\sum_{j=t+1}^{k-1}m_{h,j}+x}.
\end{equation}
Thus, taking $t=k-1$ gives the assertion (3).

We prove \eqref{*} by induction on $t$.
If $t=1$, then $x>0$ implies
$$l=\sigma_{h-1,1}+\sum_{j=1}^{k-1}m_{h,j}+x>\sigma_{h-1,1}+m_{h,1}.$$
As the largest number permuted by $w_{2,1}\cdots w_{h,1}$ is $\sigma_{h-1,1}+m_{h,1}$, we have
\begin{equation}\label{t=1}
s_l (w_{2,1}\cdots w_{h,1})=(w_{2,1}\cdots w_{h,1})s_l.
\end{equation}

Now we consider 
$s_l (w_{h+1,1}\cdots w_{n,1})$. Assume $w_{h+1,1}\neq 1$ (and so $m_{h+1,1}>0$).
Since $k>1$ and 
$\widetilde m_{h,1}+1\leq l=\sigma_{h-1,1}+\sum_{j=1}^{k-1}m_{h,j}+x<\sigma_{h-1,1}+\la_h= \sigma_{h,1}$, 
by (2), $s_l w_{h+1,1}=w_{h+1,1}s_{l+m_{h+1,1}}$ and, by an inductive argument as above,
\begin{equation*}\label{co1}
\begin{aligned}
s_l w_{h+1,1}w_{h+2,1}\cdots w_{n,1}
=w_{h+1,1}w_{h+2,1}\cdots w_{n,1}s_{l+\sum_{i=h+1}^nm_{i,1}}.
\end{aligned}
\end{equation*}
But
$l+\sum_{i=h+1}^nm_{i,1}=\sigma_{h-1,2}+\sum_{j=2}^{k-1}m_{h,j}+x$. This proves \eqref{*} for $t=1$.

Suppose now $t>1$ and \eqref{*} is true for $t-1$. That is, assume 
\begin{equation*}
\begin{aligned}
s_{\sigma_{h-1,1}+\sum_{j=1}^{k-1}m_{h,j}+x}&(w_{2,1}\cdots w_{n,1})\cdots (w_{2,t-1}\cdots w_{n,t-1})\\
=\,&(w_{2,1}\cdots w_{n,1})\cdots (w_{2,t-1}\cdots w_{n,t-1})s_{\sigma_{h-1,t}+\sum_{j=t}^{k-1}m_{h,j}+x}.
\end{aligned}
\end{equation*}

Since $\sigma_{h-1,t}+\sum_{j=t}^{k-1}m_{h,j}+x>\sigma_{h-1,t}+m_{h,t}$ and 
$$\sigma_{h,t}=\sigma_{h-1,t}+\sum_{j=t}^nm_{h,j}>\sigma_{h-1,t}+\sum_{j=t}^{k-1}m_{h,j}+x\geq\widetilde{m}_{h,t}+1,$$
applying (2) with $l=\sigma_{h-1,t}+\sum_{j=t}^{k-1}m_{h,j}+x$ gives
$$\aligned
s_l(w_{2,t}\cdots w_{h,t}w_{h+1,t}\cdots w_{n,t})&=(w_{2,t}\cdots w_{h,t})s_l(w_{h+1,t}\cdots w_{n,t})\\
&=(w_{2,t}\cdots w_{h,t}w_{h+1,t}\cdots w_{n,t})s_{l+\sum_{i=h+1}^nm_{i,t}},
\endaligned$$
where
$$l+\sum_{i=h+1}^nm_{i,t}=\sigma_{h-1,t}+\sum_{j=t}^{k-1}m_{h,j}+x+\sum_{i=h+1}^nm_{i,t}=\sigma_{h-1,t+1}+\sum_{j=t+1}^{k-1}m_{h,j}+x.$$
This proves \eqref{*} for $t$ and, hence, (3).
\end{proof}

\begin{corollary}\label{mainw}
For $0<x\leq m_{h,k}$, $l=\sigma_{h-1,1}+\sum_{j=1}^{k-1}m_{h,j}$ with $l+x<\sigma_{h,1}$, we have
\begin{equation*}
\begin{aligned}
&s_{l+x}d_M={\left(\begin{array}{ccccccc}
w_{2,1}&\cdots&w_{2,k-1}&w_{2,k}&w_{2,k+1}&\cdots&w_{2,n-1}\\
\vdots&\vdots&\cdots&\vdots&\vdots&\cdots&\vdots\\
w_{h-1,1}&\cdots&w_{h-1,k-1}&w_{h-1,k}&w_{h-1,k+1}&\cdots&w_{h-1,n-1}\\
w_{h,1}&\cdots&w_{h,k-1}&w^*_{h,k}&w_{h,k+1}&\cdots&w_{h,n-1}\\
w_{h+1,1}&\cdots&w_{h+1,k-1}&w_{h+1,k}&w_{h+1,k+1}&\cdots&w_{h+1,n-1}\\
\vdots&\vdots&\cdots&\vdots&\vdots&\cdots&\vdots\\
w_{n,1}&\cdots&w_{n,k-1}&w_{n,k}&w_{h,k+1}&\cdots&w_{n,n-1}\\
\end{array}\right),}
\end{aligned}
\end{equation*}
where $w^*_{h,k}=s_{\sigma_{h-1,k}+x}w_{h,k}$. In particular, $s_{l+1}s_{l+2}\cdots s_{l+m_{h,k}}d_M$ can be expressed by the same matrix with $w^*_{h,k}=w^+_{h,k}$, the element defined in \eqref{w+ij}.
\end{corollary}

The next result is the key to establish the decomposition in Theorem \ref{prod coset} and the multiplication formulas in Theorem \ref{KMF}. 

\begin{proposition} \label{case1} Maintain the notation as given in \eqref{notn1} and Theorem \ref{prod coset}, and
let $a=\sum_{j=1}^{k-1}m_{h+1,j}$, and $b=\sum_{j=k+1}^nm_{h,j}$.
\begin{itemize}
\item[(1)] If $m_{h+1,k}\geq 1$ then, for 
$\lambda^+=\la^{[h^+]}=\lambda+\bse_h-\bse_{h+1}$ and $0\leq p< m_{h+1,k}$, 
\begin{equation*}\label{xx}
\aligned
s_{\widetilde{\lambda}_h+1}s_{\widetilde{\lambda}_h+2}\cdots s_{\widetilde{\lambda}_h+a+p}d_M&
=s_{{\widetilde\lambda}_h^+-1}s_{{\widetilde\lambda}_h^+-2}\cdots s_{{\widetilde\lambda}_h^+-b}d_{M^+_{h,k}}(s_{\widetilde{m}_{h,k}+1}\cdots s_{\widetilde{m}_{h,k}+p})\\
&=s_{\widetilde{\lambda}_h}s_{\widetilde{\lambda}_h-1}\cdots s_{\widetilde{\lambda}_h-b+1}d_{M^+_{h,k}}(s_{\widetilde{m}_{h,k}+1}\cdots s_{\widetilde{m}_{h,k}+p}).
\endaligned
\end{equation*}
\item[(2)] If $m_{h,k}\geq 1$ then, for 
$\lambda^-=\la^{[h^-]}=\lambda-\bse_h+\bse_{h+1}$ and $q=m_{h,k}-p$ with $0< p\leq m_{h,k}$ (so
$0\leq q<m_{h,k}$), 
\begin{equation*}
\begin{aligned}
s_{\widetilde{\lambda}_h-1}s_{\widetilde{\lambda}_h-2}\cdots s_{\widetilde{\lambda}_{h}-b-q}d_M
&=s_{\widetilde{\lambda}_h^-+1}s_{\widetilde{\lambda}_h^-+2}\cdots s_{\widetilde{\lambda}_h^-+a}d_{M^-_{h,k}}(s_{\widetilde{m}_{h,k}-1}s_{\widetilde{m}_{h,k}-2}\cdots s_{\widetilde{m}_{h,k}-q})\\
&=s_{\widetilde{\lambda}_h}s_{\widetilde{\lambda}_h+1}\cdots s_{\widetilde{\lambda}_h+a-1}d_{M^-_{h,k}}(s_{\widetilde{m}_{h,k}-1}s_{\widetilde{m}_{h,k}-2}\cdots s_{\widetilde{m}_{h,k}-q}).\\
\end{aligned}
\end{equation*}
\end{itemize}
Here every product of the $s_i$'s
 is regarded as 1 if its ``length'' is 0.
\end{proposition}

\begin{proof} We only prove (1), (2) follows from (1) with a similar argument.
 We first assume that $p=0$.  
In this case, we want to prove 
\begin{equation}\label{p=0}
s_{\widetilde{\lambda}_h+1}s_{\widetilde{\lambda}_h+2}\cdots s_{\widetilde{\lambda}_h+a}d_M=
s_{\widetilde{\lambda}^+_h-1}s_{\widetilde{\lambda}^+_h-2}\cdots s_{\widetilde{\lambda}^+_h-b}d_{M^+_{h,k}}.\end{equation}

Since $a=m_{h+1,1}+\cdots+m_{h+1,k-1}$,
repeatedly applying Corollary \ref{mainw} (with $h$ replaced by $h+1$, noting $m_{h+1,k}>0$) 
yields
\begin{equation*}\label{sdm}
\begin{aligned}
s_{\widetilde{\lambda}_h+1}s_{\widetilde{\lambda}_h+2}\cdots s_{\widetilde{\lambda}_h+a}d_M={\left(\begin{array}{ccccccc}
w_{2,1}&\cdots&w_{2,k-1}&w_{2,k}&\cdots&w_{2,n-1}\\
\vdots&\vdots&\cdots&\vdots&\cdots&\vdots\\
w_{h,1}&\cdots&w_{h,k-1}&w_{h,k}&\cdots&w_{h,n-1}\\
w^+_{h+1,1}&\cdots&w^+_{h+1,k-1}&w_{h+1,k}&\cdots&w_{h+1,n-1}\\
w_{h+2,1}&\cdots&w_{h+2,k-1}&w_{h+2,k}&\cdots&w_{h+2,n-1}\\
\vdots&\vdots&\cdots&\vdots&\cdots&\vdots\\
w_{n,1}&\cdots&w_{n,k-1}&w_{n,k}&\cdots&w_{n,n-1}\\
\end{array}\right).}
\end{aligned}
\end{equation*}
(Note that, if $k=1$, then $a=0$ and so LHS of \eqref{p=0} $=d_M$. Note also that 
$w_{h+1,j}^+=1$ if $m_{h+1,j}=0$.)
By comparing this with the ``matrix'' of $d_{M^+_{h,k}}$, we now show that multiplying $d_{M^+_{h,k}}$ by $s_{\widetilde{\lambda}^+_h-1}s_{\widetilde{\lambda}^+_h-2}\cdots s_{\widetilde{\lambda}^+_h-b}$ on the left will turn the product $w^\bullet_{h,k}w^\circ_{h+1,k}$ into $w_{h,k}w_{h+1,k}$.

 If $b=0$, then $\sigma_{h,k}=\sigma_{h-1,k}+m_{h,k}$ and so $w^\bullet_{h,k}w^\circ_{h+1,k}=w_{h,k}w_{h+1,k}$ (cf. Lemma \ref{d_M}). This proves \eqref{p=0} in this case. Assume now $b>0$.
Observe that, for $\la^+=\ro(M^+_{h,k})$, $\widetilde{\lambda}^+_h-\sum_{j>k}m_{h,j}=\widetilde{\lambda}_{h-1}+\sum_{j=1}^km_{h,j}+1$. Let $l=\widetilde{\lambda}_{h-1}+\sum_{j=1}^{k-1}m_{h,j}$ and $1\leq x\leq m_{h,k}$. Then $l+x<l+x+m_{h+1,k}\leq\la_{h+1}$. By Lemma \ref{reflection}(3) (cf. \eqref{*}),
 \begin{equation}\label{k-1column}
 s_{l+x}\Pi^+_{k-1}=\Pi^+_{k-1}s_{\sigma_{h-1,k}+x},
 \end{equation}
where $\Pi^+_{k-1}$ is the product of the first $k-1$ columns of $d_{M^+_{h,k}}$.
By \eqref{d_Mm} for ${M^+_{h,k}}$ and noting \eqref{bulcir},
\begin{equation}\label{***}
\begin{aligned}
s_{\widetilde{\lambda}^+_h-1}s_{\widetilde{\lambda}^+_h-2}\cdots &s_{\widetilde{\lambda}^+_h-\sum_{j>k}m_{h,j}}d_{M^+_{h,k}}\\
=\;&\Pi^+_{k-1}\cdot s_{\sigma_{h,k}}s_{\sigma_{h,k}-1}\cdots s_{\sigma_{h-1,k}+m_{h,k}+1}\\
&(w_{2,k}\cdots w_{h-1,k}w^\bullet_{h,k}w^\circ_{h+1,k} w_{h+2,k}\cdots w_{n,k})\\
&\cdots\cdots\\
&(w_{2,n-1}\cdots w_{h,n-1}w_{h+1,n-1} w_{h+2,n-1}\cdots w_{n,n-1}).\\
\end{aligned}
\end{equation}

Since the smallest number permuted by $s_{\sigma_{h,k}}s_{\sigma_{h,k}-1}\cdots s_{\sigma_{h-1,k}+m_{h,k}+1}$ is 
$\sigma_{h-1,k}+m_{h,k}+1$, while the largest number permuted by $w_{2,1}\cdots w_{h-1,k}w_{h,k}$ is $\sigma_{h-1,k}+m_{h,k}$, it follows that $s_{\sigma_{h,k}}s_{\sigma_{h,k}-1}\cdots s_{\sigma_{h-1,k}+m_{h,k}+1}$ commutes with
$w_{2,k}\cdots w_{h-1,k}$ and $w_{h,k}$. Thus, 
$$\begin{aligned}
&\quad\;s_{\sigma_{h,k}}s_{\sigma_{h,k}-1}\cdots s_{\sigma_{h-1,k}+m_{h,k}+1}w^\bullet_{h,k}w^\circ_{h+1,k}\\
&=w_{h,k}(s_{\sigma_{h,k}}\cdots s_{\sigma_{h-1,k}+m_{h,k}+1})
s_{\sigma_{h-1,k}+m_{h,k}}s_{\sigma_{h-1,k}+m_{h,k}-1}\cdots s_{\widetilde{m}_{h,k}+1}w^\circ_{h+1,k}\\
&=w_{h,k}(s_{\sigma_{h,k}}s_{\sigma_{h,k}-1}\cdots s_{\widetilde{m}_{h,k}+1})w^\circ_{h+1,k}\\
&=w_{h,k}w_{h+1,k}.
\end{aligned}
$$
Hence, $s_{\widetilde{\lambda}^+_h-1}s_{\widetilde{\lambda}^+_h-2}\cdots s_{\widetilde{\lambda}^+_h-\sum_{j>k}m_{h,j}}d_{M^+_{h,k}}
=\;$LHS, proving the $p=0$ case.

Assume now $p>0$. Then one can easily prove by Corollary \ref{mainw} that
$$s_{l+1}\cdots s_{l+p}d_M=d_Ms_{\widetilde{m}_{h,k}+1}s_{\widetilde{m}_{h,k}+2}\cdots s_{\widetilde{m}_{h,k}+p}.$$
Now the required formula follows from \eqref{p=0}.
\end{proof}

\vspace{.3cm}
\noindent
{\it Proof of Theoren \ref{prod coset}.}
Set $D^+_h=\diag(\lambda-\bse_{h+1})+E_{h,h+1}$. Then $\ro(D^+_h)=\lambda^{[h^+]}$, $\co(D^+_h)=\lambda$, and 
$$\nu':=\nu_{D^+_h}=(\lambda_1,\lambda_2,\cdots,\lambda_h,1,\lambda_{h+1}-1,\lambda_{h+2},\cdots,\lambda_n).$$ 
Note that in this case $d_{D^+_h}=1$. Observe that
 \begin{equation}\label{nu'la}
 \mathcal{D}_{\nu'}\cap\fS_\lambda=\{1,s_{\widetilde{\lambda}_h+1},s_{\widetilde{\lambda}_h+1}s_{\widetilde{\lambda}_h+2},\cdots,s_{\widetilde{\lambda}_h+1}s_{\widetilde{\lambda}_h+2}\cdots s_{\widetilde{\lambda}_h+\lambda_{h+1}-1}\}.
 \end{equation}
Putting $d_i=s_{\widetilde{\lambda}_h+1}s_{\widetilde{\lambda}_h+2}\cdots s_{\widetilde{\lambda}_h+i}$ for $0\leq i\leq \lambda_{h+1}-1$, the left hand side becomes $\bigcup_i\fS_{\lambda^{[h^+]}} d_id_M\fS_\mu$. Since
$\lambda_{h+1}=\sum_{k; m_{h+1,k}\geq1}m_{h+1,k}$,
the first decomposition follows from Proposition \ref{case1}(1). The second decomposition can be proved similarly.
\hfill \qed

\section{Regular representation of the $q$-Schur superalgebra}
We now use Proposition \ref{case1} to derive certain multiplication formulas in $\sS(m|n,r)$ and the matrix representation of the regular representation. For any integers $0\leq t\leq s$, define Gaussian polynomials in $\sZ=\mathbb Z[\up,\up^{-1}]$ by
$$\left[\!\!\left[s\atop t\right]\!\!\right]=\left[\!\!\left[s\atop t\right]\!\!\right]_\bsq=\frac{[\![s]\!]^!}{[\![t]\!]^![\![s-t]\!]^!},$$
where $[\![r]\!]^{!}:=[\![1]\!][\![2]\!]\cdots[\![r]\!]$ with $[\![i]\!]=1+\bsq+\cdots+\bsq^{i-1}$ ($\bsq=\up^2$). Define $[r]^!$ similarly with $[i]=\frac{\up^i-\up^{-i}}{\up-\up^{-1}}$.

For $\lambda\in\Lambda(m|n,r)$, denote $\mathcal{P}_{W_\lambda}$ to be the {\it super} Poincar\'e polynomial
\begin{equation}\label{superP}
\mathcal{P}_{W_\lambda}=\sum_{w_0\in W_{\la^{(0)}},w_1\in W_{\la^{(1)}}}(\bsq)^{\ell(w_0)}(\bsq^{-1})^{\ell(w_1)}.
\end{equation}

For $1\leq h\leq m+n$, define $\dq_h,\ddq_h,\up_h$ by
$$\begin{cases}
\dq_h=1, &\ddq_h=\bsq, \quad\;\;\up_h=\up,      \quad          \text{ if }1\leq h\leq m;\\
\dq_h=-\bsq^{-1},&\ddq_h=-1, \quad \up_h=\up^{-1},\text{ if }m<h\leq m+n,
\end{cases}
$$
and let $\bsq_h=\up_h^2$. Recall the basis $\{\phi_A\}_{A\in M(m|n,r)}$ given in Lemma \ref{DR5.8}.

\begin{theorem}\label{KMF}
For any $A=(a_{i,j})\in M(m|n,r)$ and $1\leq h< m+n$, let $D_h^+,D_h^-$ be the matrices defined by the conditions that $D_h^+-E_{h,h+1}, D_h^--E_{h+1,h}$ are diagonal and $\co(D_h^+)=\co(D_h^-)=\ro(A)$, and assume $D_h^+, D_h^-\in M(m|n,r)$. Then the following multiplication formulas hold in $\sS(m|n,r)$:
$$\aligned
(1)\quad&\phi_{D^+_h}\phi_A=\sum_{k\in[1,m+n]\atop a_{h+1,k}\geq 1}\dq_{h+1}^{\sum_{j<k}a_{h+1,j}}\ddq_h^{\sum_{j>k}a_{h,j}}[\![a_{h,k}+1]\!]_{\bsq_h}\phi_{A^+_{h,k}};\\
(2)\quad&\phi_{D^-_h}\phi_A=\sum_{k\in[1,m+n]\atop a_{h,k}\geq 1}\dq_h^{\sum_{j>k}a_{h,j}}\ddq_{h+1}^{\sum_{j<k}a_{h+1,j}}[\![a_{h+1,k}+1]\!]_{\bsq_{h+1}}\phi_{A^-_{h,k}}.\endaligned
$$
{\rm(Here $[1,m+n]=\{1,2,\ldots,m+n\}$.)}
\end{theorem}
\begin{proof} We only prove (1). The proof of (2) is symmetric.

Let $\lambda=\ro(A)$, $\mu=\co(A)$, $d=d_A$ and $W_\nu=\fS^d_\lambda\cap \fS_\mu=W_{\nu^{(0)}}\times W_{\nu^{(1)}}$ (cf. \eqref{notn1}), where $W_{\nu^{(i)}}=W^d_{\lambda^{(i)}}\cap W_{\mu^{{(i)}}}$ for $i=0,1$. Then $\lambda=\co(D^+_h)$, $\la^{[h^+]}=\ro(D^+_h)=\lambda+\bse_h-\bse_{h+1}$, and $\jmath(\la^{[h^+]},1,\lambda)=D^+_h$.

Putting $W_{\nu'(h)}=W_{\la^{[h^+]}}\cap W_\la$, we see from \eqref{nu'la},
$$\mathcal{D}_{\nu'(h)}\cap W_\la=\{1,s_{\widetilde{\lambda}_h+1},s_{\widetilde{\lambda}_h+1}s_{\widetilde{\lambda}_h+2},
\cdots,s_{\widetilde{\lambda}_h+1}\cdots s_{\widetilde{\lambda}_h+\lambda_{h+1}-1}\}.$$
Since $\mathcal{D}_{\nu'(h)}\cap W_\la\subseteq W_{\la^{(1)}}$ whenever $h\geq m$, the element
$T_{\mathcal{D}_{\nu'(h)}\cap W_\la}$ used in \eqref{double coset} can be written as
$T_{\mathcal{D}_{\nu'(h)}\cap W_\la}=\sum_{w\in \mathcal{D}_{\nu'(h)}\cap W_\la}(\dq_{h+1})^{\ell(w)}T_w.$

By definition, to compute $\phi_{D^+_h}\phi_A$, it suffices to write $\phi_{D^+_h}\phi_A(\xy_{\mu})$ as a linear combination of some $T_{W_\xi d'W_\mu}$, where $\xi=\la^{[h^+]}$. We compute this within $\sS_{\mathbb Q(\up)}(m|n,r)$:
\begin{equation}
\begin{aligned}
\phi_{D^+_h}\phi_A(\xy_{\mu})&=\phi^1_{\xi,\lambda}\phi^d_{\lambda,\mu}(\xy_\mu)=\phi^1_{\xi,\lambda}(T_{W_\lambda d W_{\mu}})\\
&=\phi^1_{\xi,\lambda}(\xy_{\lambda} T_d T_{\mathcal{D_\nu}\cap W_\mu})\mbox{ (by \eqref{double coset})}\\
&=T_{W_\xi W_\lambda}T_d T_{\mathcal{D_\nu}\cap W_\mu}=(\mathcal{P}_{W_\nu})^{-1}T_{W_\xi W_\lambda}T_d \xy_\mu\\
&=(\mathcal{P}_{W_\nu})^{-1}\xy_\xi T_{\mathcal{D}_{\nu'(h)}\cap W_\lambda}T_d\xy_\mu\\
&=(\mathcal{P}_{W_\nu})^{-1}\sum_{w\in \mathcal{D}_{\nu'(h)}\cap W_\lambda}\xy_\xi ({\dq_{h+1}}^{\ell(w)}T_w) T_d \xy_\mu.\\
\end{aligned}
\end{equation}

Note that $d=d_A\in \mathcal{D}_{\lambda\mu}$. If $a_{h+1,k}>0$ and 
$w_p:=s_{\widetilde{\lambda}_h+1}s_{\widetilde{\lambda}_h+2}\cdots s_{\widetilde{\lambda}_h+\sum_{j=1}^{k-1}a_{h+1,j}+p}$ for some $0\leq p<a_{h+1,k}$, then by Proposition \ref{case1}(1), we have
$$w_pd=s_{{\widetilde\lambda}_h}s_{{\widetilde\lambda}_h-1}\cdots s_{{\widetilde\lambda}_h-\sum_{j=k+1}^{m+n}a_{h,j}+1}d^+(s_{\widetilde{a}_{h,k}+1}\cdots s_{\widetilde{a}_{h,k}+p}),$$
where $d^+=d_{A^+_{h,k}}$. Clearly, $\sum_{j<k}a_{h+1,j}=\ell(w_p)-p$.
If we put $Q_{h+1,k}=\dq_{h+1}^{\sum_{j<k}a_{h+1,j}}$, then
$$\aligned
\sum_{p=0}^{a_{h+1,k}-1}&{\dq_{h+1}}^{\ell(w_p)}T_{w_p} T_d=Q_{h+1,k}
T_{\widetilde{\lambda}_h}T_{\widetilde{\lambda}_h-1}\cdots T_{\widetilde{\lambda}_h-\sum_{j>k}a_{h,j}+1}T_{d^+}\\
&\cdot(1+\dq_{h+1}T_{\widetilde{a}_{h,k}+1}+\cdots+\dq_{h+1}^{a_{h+1,k}-1}T_{\widetilde{a}_{h,k}+1}\cdots T_{\widetilde{a}_{h,k}+a_{h+1,k}-1}).
\endaligned
$$
Thus,
\begin{equation*}
\begin{aligned}
&\sum_{w\in \mathcal{D}_{\nu'}\cap W_\lambda}\xy_\xi (\dq_{h+1}^{\ell(w)}T_w T_d) \xy_\mu\\
=&\sum_{k\in[1,m+n]\atop a_{h+1,k}\geq 1} Q_{h+1,k} \xy_{\xi}
T_{\widetilde{\lambda}_h}T_{\widetilde{\lambda}_h-1}\cdots T_{\widetilde{\lambda}_h-\sum_{j>k}a_{h,j}+1}T_{d^+}\\
&\cdot(1+(\dq_{h+1})T_{\widetilde{a}_{h,k}+1}+\cdots+(\dq_{h+1})^{a_{h+1,k-1}}T_{\widetilde{a}_{h,k}+1}\cdots T_{\widetilde{a}_{h,k}+a_{h+1,k}-1})\xy_\mu.
\end{aligned}
\end{equation*}
Since 
$$
\begin{aligned}
&(1+(\dq_{h+1})T_{\widetilde{a}_{h,k}+1}+\cdots+(\dq_{h+1})^{a_{h+1,k-1}}T_{\widetilde{a}_{h,k}+1}\cdots T_{\widetilde{a}_{h,k}+a_{h+1,k}-1})\xy_\mu\\
&=(1+\dq_{h+1}\ddq_k+\cdots+(\dq_{h+1}\ddq_k)^{a_{h+1,k}-1})\xy_\mu\\
&=[\![a_{h+1,k}]\!]_{\dq_{h+1}\ddq_k}\xy_\mu
\end{aligned}
$$
and 
$$\xy_\xi T_{\widetilde{\lambda}_h}T_{\widetilde{\lambda}_h-1}\cdots T_{\widetilde{\lambda}_h-\sum_{j>k}a_{h,j}+1}
=\ddq_h^{\sum_{j>k}a_{h,j}}\xy_\xi,
$$
it follows that
\begin{equation*}
\begin{aligned}
\phi_{D^+_h}\phi_A(\xy_\mu)&=\mathcal{P}_{W_\nu}^{-1}\sum_{a_{h+1,k}\geq1}Q_{h+1,k}\ddq_h^{\sum_{j>k}a_{h,j}}
[\![a_{h+1,k}]\!]_{\dq_{h+1}\ddq_k}\xy_\xi T_{d^+}\xy_\mu\\
&=\sum_{a_{h+1,k}\geq1}\frac{\mathcal{P}_{W_{\nu''}}}{\mathcal{P}_{W_\nu}}Q_{h+1,k}
\ddq_h^{\sum_{j>k}a_{h,j}}[\![a_{h+1,k}]\!]_{\dq_{h+1}\ddq_k}T_{W_\xi d^+W_\mu}\\
&=\sum_{a_{h+1,k}\geq1}\frac{\mathcal{P}_{W_{\nu''}}}{\mathcal{P}_{W_\nu}}Q_{h+1,k}
\ddq_h^{\sum_{j>k}a_{h,j}}[\![a_{h+1,k}]\!]_{\dq_{h+1}\ddq_k}
\phi_{A^+_{h,k}}(\xy_{\mu}),
\end{aligned}
\end{equation*}
where $\nu''=\nu_M$ with $M=A^+_{h,k}$ or $W_{\nu''}=W_\xi^{d^+}\cap W_\mu$.
Hence, noting
$$\frac{\mathcal{P}_{W_{\nu''}}}{\mathcal{P}_{W_\nu}}=\frac{[\![a_{h,k}+1]\!]_{\bsq_k}^![\![a_{h+1,k}-1]\!]_{\bsq_k}^!}
{[\![a_{h,k}]\!]_{\bsq_k}^![\![a_{h+1,k}]\!]_{\bsq_k}^!}=\frac{[\![a_{h,k}+1]\!]_{\bsq_k}}{[\![a_{h+1,k}]\!]_{{\bsq_k}}},$$
 we obtain
\begin{equation}
\begin{aligned}
\phi_{D^+_h}\phi_A=\sum_{k\atop a_{h+1,k}\geq1}\dq_{h+1}^{\sum_{j<k}a_{h+1,j}}\ddq_h^{\sum_{j>k}a_{h,j}} \frac{[\![a_{h,k}+1]\!]_{\bsq_k} [\![a_{h+1,k}]\!]_{\dq_{h+1}\ddq_k}}{[\![a_{h+1,k}]\!]_{{\bsq_k}}} \phi_{A^+_{h,k}}.
\end{aligned}
\end{equation}
It remains to prove that
\begin{equation}\label{str const}
\frac{[\![a_{h,k}+1]\!]_{\bsq_k} [\![a_{h+1,k}]\!]_{\dq_{h+1}\ddq_k}}{[\![a_{h+1,k}]\!]_{{\bsq_k}}}=[\![a_{h,k}+1]\!]_{\bsq_h}.
\end{equation}
This can be seen in cases. For example,
if $h< m$ and $k\leq m$ (resp.,  $h>m$ and $k>m$), then $\dq_{h+1}=1$, $\ddq_k=\bsq$ (resp., $\dq_{h+1}=-\bsq^{-1}$, $\ddq_k=-1$), and so $\bsq_k=\bsq_h$ (resp., $\dq_{h+1}\ddq_k=\bsq_h$). 
Hence,
$$\frac{[\![a_{h,k}+1]\!]_{\bsq_k} [\![a_{h+1,k}]\!]_{\dq_{h+1}\ddq_k}}{[\![a_{h+1,k}]\!]_{{\bsq_k}}}=[\![a_{h,k}+1]\!]_{\bsq_h}.$$
When $h\leq m$ and $k>m$, or  $h>m$ and $k\leq m$, we must have $a_{h,k}+1=a_{h+1,k}=1$. Thus,
$ [\![a_{h,k}+1]\!]_{\bsq_k}= [\![a_{h,k}+1]\!]_{\dq_{h+1}\ddq_k}=[\![a_{h+1,k}]\!]_{{\bsq_k}}=1= [\![a_{h,k}+1]\!]_{\bsq_h}.$ Finally,
when $h=m$ and $k\leq m$, we have $\bsq_h=\bsq_k$ and  $\dq_{h+1}\ddq_k=-\bsq^{-1}\bsq=-1$. But
$a_{h+1,k}=a_{m+1,k}=1$, forcing
$ [\![a_{h+1,k}]\!]_{\dq_{h+1}\ddq_k}=[\![a_{h+1,k}]\!]_{{\bsq_k}}=1$. Hence,
$$\frac{[\![a_{h,k}+1]\!]_{\bsq_k} [\![a_{h+1,k}]\!]_{\dq_{h+1}\ddq_k}}{[\![a_{h+1,k}]\!]_{{\bsq_k}}}=[\![a_{h,k}+1]\!]_{\bsq_h},$$
proving \eqref{str const} and, hence, formula (1).
\end{proof}
 
If $n=0$, then $\sS(m|0,r)$ is the usual $\bsq$-Schur algebra which is defined in \cite{BLM} as a convolution algebra of the $m$-step flags of an $r$-dimensional space. Similar multiplication formulas are obtained in loc. cit. by counting intersections of certain orbits. Observe that, for $h<m$, 
 $$\dq_{h+1}^{\sum_{j<k}a_{h+1,j}}\ddq_h^{\sum_{j>k}a_{h,j}}=\bsq^{\sum_{j>k}a_{h,j}},\qquad\dq_h^{\sum_{j>k}a_{h,j}}\ddq_{h+1}^{\sum_{j<k}a_{h+1,j}}=\bsq^{\sum_{j<k}a_{h+1,j}}.$$
  
\begin{corollary}\label{KMFcor}
The multiplication formulas in Theorem \ref{KMF} for $\mathcal S(m|0,r)$ coincide with the ones in \cite[Lemma 3.4]{BLM}.
\end{corollary}

We now make a comparison of these new formulas with ones given in \cite[Lemma 3.1]{DG}, derived through the relative norm method.

The $\sH$-module $\fT(m|n,r)$ is isomorphic to the tensor superspace $V(m|n)^{\otimes r}$ (over $\sZ$!) with an $\sH$-action defined in \cite[(1.0.10)]{DG}; see \cite[Proposition 8.3]{DR}. In fact, the endomorphism algebra of  $V(m|n)^{\otimes r}$ has a relative norm basis $\{N_A\}_{A\in M(m|n,r)}$ acting on the right. Matrix transposing may turn the right action to a left action and result in a basis denoted by $\{\zeta_A\}_{A\in M(m|n,r)}$. The $\sH$-module isomorphism induces an algebra isomorphism (cf. \cite[Corollary 8.4]{DR} and  \cite[Lemma 2.3]{DG2})
$$\End_{\sH}(V(m|n)^{\otimes r})^{\text{\rm op}}\longrightarrow \sS(m|n,r),\zeta_A\longmapsto (-1)^{\widehat A}\phi_A,$$
where  $\widehat{A}=\sum_{m<k<i\leq m+n,\\1\leq j<l\leq m+n}a_{i,j}a_{k,l}.$
\begin{corollary}\label{coincide}Let 
$$ f^+_{h,k}(\bsq,A)=\dq_{h+1}^{\sum_{j<k}a_{h+1,j}}\ddq_h^{\sum_{j>k}a_{h,j}}, \quad f^-_{h,k}(\bsq,A)=\dq_h^{\sum_{j>k}a_{h,j}}\ddq_{h+1}^{\sum_{j<k}a_{h+1,j}}.$$ Then 
$$(-1)^{\widehat{D}^+_h+\widehat A+\widehat{A}^+_{h,k}} f^+_{h,k}(\bsq,A)=f_k(\bsq,A,h)\text{ and
}(-1)^{\widehat{D}^-_h+\widehat A+\widehat{A}^-_{h,k}} f^-_{h,k}(\bsq,A)=g_k(\bsq,A,h),$$ where $f_k(\bsq,A,h)$ and $g_k(\bsq,A,h)$ are defined in \cite[(3.0.1-2)]{DG}. In particular, rewriting the multiplication formulas in Theorem \ref{KMF} in terms of the $\zeta$-basis results in the formulas in \cite[Lemma 3.1]{DG}. 
\end{corollary}
\begin{proof} We have 
\begin{equation}\label{Jie}
f^+_{h,k}(\bsq,A)=
\begin{cases}
\bsq^{\sum_{j>k}a_{h,j}}, &\text{ if }h<m;\\
(-1)^{\sum_{j<k}a_{m+1,j}}\bsq^{-\sum_{j<k}a_{m+1,j}+\sum_{j>k}a_{m,j}},&\text{ if }h=m;\\
(-1)^{\sum_{j<k}a_{h+1,j}+\sum_{j>k}a_{h,j}}\bsq^{-\sum_{j<k}a_{h+1,j}},&\text{ if }h>m.\\
\end{cases}
\end{equation}
On the other hand (cf. \cite[Lemma 5.1]{DG}), for the choice of $+$ or $-$,
\begin{equation*}
\widehat{D}^\pm_{h}+\widehat{A}+\widehat{A}^\pm_{h,k}=\left\{
\begin{aligned}
&2\widehat{A} &\mbox{ if } h<m;\\
&\mp\sum_{i>m+1,j<k}a_{i,j}+2\widehat{A} &\mbox{ if } h=m;\\
&\mp\sum_{j>k}a_{h,j}\pm\sum_{j<k}a_{h+1,j}+2\widehat{A}  &\mbox{ if } h>m.
\end{aligned}
\right.
\end{equation*}
Adjusting the right hand side of \eqref{Jie} by the corresponding sign for the ``$+$'' case gives $f_k(\bsq,A,h)$. The ``$-$'' case is similar.
\end{proof}

Theorem \ref{KMF} and Corollary \ref{coincide} give a new method to derive  the key fundamental multiplication formulas given in \cite[Lemma 3.1]{DG}. 

By introducing the normalised basis $\{[A]\}_{A\in M(m|n,r)}$, where\footnote{The element $[A]$ is denoted by $\xi_A$ in \cite[(4.2.1)]{DG}.}
 $$[A]=(-1)^{\widehat{A}}\up^{-d(A)}\phi_A\;\text{ with }\; d(A)=\sum_{i>k,j<l}a_{i,j}a_{k,l}+\sum_{j<l}(-1)^{\widehat{i}}a_{i,j}a_{i,l},$$ we may modify the formulas given in Theorem \ref{KMF} to obtain further multiplication formulas for the $[\;\;]$-basis; cf. (the $p=1$ case of) \cite[Propositions 4.4\&4.5]{DG}. 
 

\begin{corollary}\label{KyMF}
Maintain the notation above and let $\epsilon_{h,k}=0$ for $h\neq m$, and $\epsilon_{m,k}=
\sum_{i>m,j<k}a_{i,j}$. The following multiplication formulas hold in $\sS_R(m|n,r)$: 
\begin{itemize}
\item[(1)] $[D^+_h][A]=\displaystyle\sum_{k\in[1,m+n] \atop a_{h+1,k}\geq 1}(-1)^{\epsilon_{h,k}}\up_h^{f^+_{h,k}}\overline{[\![a_{h,k}+1]\!]}_{\up_h^2}[A^+_{h,k}]$,\\ where $f^+_{h,k}=\sum_{j\geq k}a_{h,j}-(-1)^{\widehat h+\widehat{h+1}}\sum_{j>k}a_{h+1,j}$;
\item[(2)] $[D^-_h][A]=\displaystyle\sum_{k\in[1,m+n]\atop a_{h,k}\geq 1}(-1)^{\epsilon_{h,k}}\up_{h+1}^{f^-_{h,k}}\overline{[\![a_{h+1,k}+1]\!]}_{\up_{h+1}^2}[A^-_{h,k}],$\\ where  $f^-_{h,k}=\sum_{j\leq k}a_{h+1,j}-(-1)^{\widehat h+\widehat{h+1}}\sum_{j<k}a_{h,j}$.
\end{itemize}

\end{corollary}

 
The first important application of the multiplication formulas above is a new realisation of the quantum supergroup ${\bf U}_\up(\mathfrak{gl}_{m|n})$; see the argument from \cite[\S5]{DG} onwards and, in particular, see \cite[Definition 6.1, Theorem 8.4]{DG}.

We now  seek further applications of these multiplication formulas. 

We will show below that the formulas provide enough information for the regular representation of the integral $q$-Schur superalgebra $\sS_R(m|n,r)$. We then use such a representation to determine the semisimplicity of $q$-Schur superalgebras and to construct infinitesimal and little ones without involving the quantum supergroup or quantum coordinate superalgebra.

We return to the general setting for $\sS_R(m|n,r)$ defined relative to a commutative ring $R$ and an invertible parameter $\ups\in R$ or $q=\ups^2$. Base change via $\sZ\to R, \up\mapsto\ups$, we may turn the multiplication formulas in $\sS(m|n,r)$ into similar formulas in $\sS_R(m|n,r)$. In fact, these formulas can be interpreted as the matrix representation of certain generators for $\sS_R(m|n,r)$ relative to the basis $\{[A]\}_{A\in M(m|n,r)}$. 

Let
$$M(m|n)^\pm=\{A=(a_{i,j})\in M(m|n)\mid a_{i,i}=0,
1\leq i\leq m+n\}.$$
For $A\in M(m|n)^{\pm}$ and
$\bsj=(j_1,j_2,\cdots,j_{m+n})\in\mathbb{Z}^{m+n}$, define
\begin{equation}\label{Ajr}
A(\bsj,r)=\begin{cases}\sum_{\substack{\lambda\in\Lambda(m|n,r-|A|)}}(-1)^{\overline{A+\lambda}}\ups^{\lambda*\bsj}[A+\lambda],&\text{ if }|A|\leq r;\\
0,&\text{ otherwise,}\end{cases}
\end{equation}
where $\lambda*\bsj=\sum_{i=1}^{m+n}(-1)^{\widehat{i}}\lambda_ij_i$ is the super (or signed) ``dot product'', $A+\la=A+\diag(\la)$ and $\overline{M}=\sum_{\substack{m+n\geq i> m\geq k\geq1
\\m<j<l\leq m+n}}m_{i,j}m_{k,l}$ for a matrix $M$. We also let $1_\la=[\diag(\la)]$ for all $\la\in\La(m|n,r)$, the identity map on $\xy_\la\sH_R$. Then
$1_\la[A]=\delta_{\la,\ro(A)}[A].$
For the zero matrix $O$, $\bse_i\in\La(m|n,1)$ and $p\geq1$, set
$$\sck_i=O(\bse_i,r),\quad \sce_h^{(p)}=(pE_{h,h+1})(\mathbf0,r),\quad \scf_h^{(p)}=(pE_{h+1,h})(\mathbf0,r).$$
Note that $\sck_i=\sum_{\la\in\La(m|n,r)}\ups^{(-1)^{\widehat i}\la_i}1_\la$ and $\sce_m^2=0=\scf_m^2$.

Let $\sS_R^-$, $\sS_R^+$ be the subsuperalgebra of $\sS_R(m|n,r)$ generated respectively by $\scf_h^{(p)}$, $\sce_h^{(p)}$ for all $1\leq h<m+n$, $p\geq1$, and $\sS_R^0$ the subsuperalgebra spanned by all $1_\la$.

The first assertion of the following is \cite[Corollary 8.5]{DG}.
\begin{theorem}\label{KeyMF}
 The $q$-Schur superalgebra $\sS_R=\sS_R(m|n,r)$ is generated by
$\sck_i,$ $1_\la,$ $\sce_h^{(p)},\scf_h^{(p)}$ 
for all $1\leq h,i\leq m+n, h\not=m+n,$ $\la\in\La(m|n,r)$, $1\leq p\leq r$, and $\sS_R=\sS_R^+\sS_R^0\sS_R^-$.
These generateors have the following matrix representations relative to the basis $\{[A]\}_{A\in M(m|n,r)}$:
\begin{itemize}
\item[(0)] $\sck_i[A]=\ups^{(-1)^{\widehat i}\ro(A)_i}[A]$, \; $1_\la[A]=\delta_{\la,\ro(A)}[A]$;
\vspace{2ex}

\item[(1)] $\sce_h^{(p)}[A]=\displaystyle\sum_{\substack{\nu\in\Lambda(m|n,p)\\\nu\leq
 \row_{h+1}(A)}}\ups_h^{f^+_h(\nu,A)}\prod_{k=1}^{m+n}\overline{\left[\!\!\left[a_{h,k}+\nu_k\atop\nu_k\right]\!\!\right]}_{\ups_h^2}
 [A+\sum_l\nu_l(E_{h,l}-E_{h+1,l})],$\vspace{2ex}
 \item[] where $h\neq m$,
$f^+_h(\nu,A)=\sum_{j\geq
 t}a_{h,j}\nu_t-\sum_{j>t}a_{h+1,j}\nu_t+\sum_{t<t'}\nu_t\nu_{t'}$ and $\nu\leq\nu'$ means that $\nu_i\leq\nu_i'$ for all $i$;\vspace{2ex}
 
\item[(2)]$\scf_h^{(p)}[A]=\displaystyle\sum_{\substack{\nu\in\Lambda(m|n,p)\\\nu\leq
 \row_{h}(A)}}\ups_{h+1}^{f^-_h(\nu,A)}\prod_{k=1}^{m+n}\overline{\left[\!\!\left[a_{h+1,k}+\nu_k\atop\nu_k\right]\!\!\right]}_{\ups_{h+1}^2}
 [A-\sum_l\nu_l(E_{h,l}-E_{h+1,l})],$\vspace{2ex}
\item[] where $h\neq m$ and
$ f^-_h(\nu,A)=\sum_{j\leq
 t}a_{h+1,j}\nu_t-\sum_{j<t}a_{h,j}\nu_t+\sum_{t<t'}\nu_t\nu_{t'}$;\vspace{2ex}
 
 \item[(3)] $\sce_m[A]=\displaystyle\sum_{k\atop a_{m+1,k}\geq 1}(-1)^{\sum_{i>m,j<k}a_{i,j}}\ups_m^{f^+_{m,k}(A)}\overline{[\![a_{m,k}+1]\!]}_{\ups_m^2}[A^+_{m,k}],$\vspace{2ex}
 \item[] where $f^+_{m,k}(A)=\sum_{j\geq k}a_{m,j}+\sum_{j>k}a_{m+1,j}$;\vspace{2ex}
 
\item[(4)] $\scf_m[A]=\displaystyle\sum_{k\atop a_{m,k}\geq 1}(-1)^{\sum_{i>m,j<k}a_{i,j}}\ups_{m+1}^{f^-_{m,k}(A)}\overline{[\![a_{m+1,k}+1]\!]}_{\ups_{m+1}^2}[A^-_{m,k}],$ \vspace{2ex}
\item[] where $f^-_{m,k}(A)=\sum_{j\leq k}a_{m+1,j}+\sum_{j<k}a_{m,j}$. \vspace{2ex}
\end{itemize}
\end{theorem}
\begin{proof} The first assertion follows from in \cite[Corollary 8.5]{DG} (cf. \cite[Theorem 6.3]{DG}). Now the relations in (0) are clear.
Since $\sce_h^{(p)}[A]=\sce_h^{(p)}1_{\ro(A)}[A]$, $\scf_h^{(p)}[A]=\scf_h^{(p)}1_{\ro(A)}[A]$, and $\sce_h^{(p)}1_{\ro(A)}=(-1)^{\overline{D^+_{h,p}}}[D^+_{h,p}]$, $\scf_h^{(p)}1_{\ro(A)}=(-1)^{\overline{D^-_{h,p}}}[D^-_{h,p}]$, where the matrices $D^\pm_{h,p}\in M(m|n,r)$ are defined by the conditions that
$\co(D^\pm_{h,p})=\ro(A)$ and $D^+_{h,p}-pE_{h,h+1}$, $D^-_{h,p}-pE_{h+1,h}$ are diagonal, (1) and (2) follow from \cite[Proposition 4.4]{DG}\footnote{$D^+_{h,p}$, $D^-_{h,p}$ are denoted by $U_p$, $L_p$.} and \cite[Lemma 5.1(1)]{DG} which tells $\overline{D^\pm_{h,p}}=0$. The remaining (3) and (4) follow from the $h=m$ case of Corollary \ref{KyMF}; see \cite[Proposition 4.5]{DG}.
\end{proof}

Note that we have in $\sS_F(m|n,r)$
\begin{equation}\label{[E,F]}
\sce_h\scf_k-(-1)^{\widehat h\widehat k}\scf_k\sce_h=\delta_{h,k}\frac{\sck_h\sck_{h+1}^{-1}-\sck_h^{-1}\sck_{h+1}}{\ups_h-\ups_h^{-1}}.
\end{equation}

\section{Semisimple $q$-Schur superalgebras}

The most fabulous application of the multiplication formulas is the realisations of quantum $\mathfrak{gl}_n$ \cite{BLM}
and quantum super $\mathfrak{gl}_{m|n}$ \cite{DG}. We now use these formulas to construct certain modules from which we obtain a semisimplicity criterion of $q$-Schur superalgebras. {\it From now on, 
let $F$ be a field of characteristic $\neq2$ and assume that $\ups\in F^\times$ and $q=\ups^2\neq1$.} Since every simple $\sS_F(m|n,r)$-supermodule is also a simple $\sS_F(m|n,r)$-module (see e.g., \cite[Proposition 4.1]{DGW2}), we will drop the prefix ``super'' in the sequel for simplicity.

We first determine the semisimplicity for $\sS_F(1|1,r)$ (see \cite{mz} for the $q=1$ case).

\begin{lemma}\label{ss1}Assume that $q\neq1$  is a primitive $l$-th root of unity.
\begin{itemize}\item[(1)]
If $l \nmid r$  then $\sS_F(1|1, r)$ is semisimple and has
exact $r$ nonisomorphic  irreducible modules  which  are all two dimensional.

\item[(2)]
If $l \mid r$  then $\sS_F(1|1, r)$ is not semisimple and has
exact $r+1$ nonisomorphic  irreducible modules  which  are all one dimensional.
\end{itemize}
\end{lemma}

\begin{proof} Let $\sS_F=\sS_F(1|1,r)$. We first observe that
$$M(1|1,r)=\{A_a,A_b^+,A_c^-,A_d^\pm\mid a\in[0, r], b,c\in[0,r-1],d\in[0,r-2]\},$$
where $A_a,A_b^+,A_c^-,A_d^\pm$ denote respectively the following matrices
$$\begin{pmatrix}
             a & 0 \\
             0   & r-a
           \end{pmatrix},\;
         \begin{pmatrix}
             b& 1 \\
             0 & r-b-1
           \end{pmatrix},\;
           \begin{pmatrix}
             c & 0 \\
             1   & r-c-1
           \end{pmatrix},\;
           \begin{pmatrix}
             d & 1 \\
             1 & r-d-2
           \end{pmatrix}.
                 $$
Note that $1_a:=1_{(a,r-a)}=[A_a]$ and $\sum_{a=0}^r1_a$ is the identity element. So
$$\sS_F=\bigoplus_{a=0}^{r}\sS_F1_{a}\quad\text{and}\quad\dim\sS_F=4r.
$$        
Since $\sS_F1_a$ is spanned by $[A]$ with $\co(A)=(a,r-a)$, it follows that
$$\aligned
\sS_F1_0&=\text{span}\{1_0,[A_0^+]\}, \quad\sS_F1_r=\text{span}\{1_r,[A_{r-1}^-]\},\\
\sS_F1_a&=\text{span}\{1_a,[A_a^+],[A_{a-1}^-],[A_{a-1}^\pm]\}, \forall a\in[1,r-1].
\endaligned$$

By Theorem \ref{KeyMF}(3)\&(4), we have 
$$\aligned
&\sce_1[A_0^+]=0,\;\scf_1[A_0^+]=\ups^{-(r-1)}[\![r]\!]_{q}1_0,\;\sce_11_0=[A_0^+],\;\scf_11_0=0,\;\\
&\scf_1[A_{r-1}^-]=0,\;\sce_1[A^-_{r-1}]=\ups^{r-1}[\![r]\!]_{q^{-1}}1_r, \;\sce_11_r=0,\; \scf_11_r=[A_{r-1}^-].\\
\endaligned$$
If $l\nmid r$, then $\ups^{-(r-1)}[\![r]\!]_{q}=\ups^{r-1}[\![r]\!]_{q^{-1}}\neq0$ in $F$, and we see easily that $L(1):=\sS_F1_0$ is irreducible. Similarly, $L(r):=\sS_F1_r$ is irreducible if $l\nmid r$.

If $l\mid r$, then $L(1)$ is indecomposable and $[A_0^+]$ spans a submodule $\overline L(1)$ of $L(1)$. Let $\overline{L}(0)=L(1)/\overline{L}(1)$. Similarly, $[A_{r-1}^-]$ spans a submodule $\overline{L}(r-1)$. Let $\overline{L}(r)=L(r)/\overline{L}(r-1)$.

For $a\in[1,r-1]$, applying Theorem \ref{KeyMF} again yields
\begin{equation}\label{pence}
\aligned
(1)\quad&\sce_1[A_a^+]=0,\; \scf_1[A_{a}^+]=\ups^{-(r-1)}[\![r-a]\!]_q1_a+[A_{a-1}^\pm],\\
(2)\quad&\scf_1[A_{a-1}^-]=0,\; \sce_1[A_{a-1}^-]=\ups^{r-1}[\![a]\!]_{q^{-1}}1_a-[A_{a-1}^\pm],\\
(3)\quad&\sce_1[A_{a-1}^\pm]=\ups^{r-1}[\![a]\!]_{q^{-1}}[A_a^+],\;\sce_11_a=[A_a^+],\\
(4)\quad&\scf_1[A_{a-1}^\pm]=-\ups^{-(r-1)}[\![r-a]\!]_q[A_{a-1}^-],\;\scf_11_a=[A_{a-1}^-].
\endaligned
\end{equation}
Let 
$$L(a+1)=\text{span}\{[A_a^+], \scf_1[A_{a}^+]\}\text{ and }L(a)=\text{span}\{[A_{a-1}^-], \sce_1[A_{a-1}^-]\}.$$ 
If $l\nmid r$, we claim that $\sS_F1_a=L(a+1)\oplus L(a)$ is a direct sum of irreducible submodules. Indeed,
 $[\![a]\!]_{q^{-1}}$ and $[\![r-a]\!]_q$ cannot be both zero in this case. So $L(a+1)\cap L(a)=0$, forcing $\sS_F1_a=L(a+1)\oplus L(a)$ as vector spaces. Since, by \eqref{[E,F]},
\begin{equation}\label{trump}
\sce_1\scf_1[A_{a}^+]=(\sce_1\scf_1+\scf_1\sce_1)[A_{a}^+]=\frac{\sck_1\sck_2^{-1}-\sck_1^{-1}\sck_2}{\ups-\ups^{-1}}
[A_{a}^+]=\frac{\ups^r-\ups^{-r}}{\ups-\ups^{-1}}
[A_{a}^+],
\end{equation}
and $\frac{\ups^r-\ups^{-r}}{\ups-\ups^{-1}}\neq0$,  every nonzero element in $L(a+1)$ generates $L(a+1)$. Hence, $L(a+1)$ is an irreducible submodule. Likewise, $L(a)$  is a submodule. This proves that $\sS_F1_a$ is semisimple for all $a\in[1,r-1]$. Hence, $\sS_F$ is semisimple.

 Assume now $l\mid r$. Then, by \eqref{trump}, $\sce_1(\scf_1[A_{a}^+])=0$. On the other hand, $\scf_1^2=0$ implies 
 $\scf_1(\scf_1[A_{a}^+])=0.$
 Thus, $\scf_1[A_{a}^+]$ spans a submodule $\overline{L}(a)$ of $L(a+1)$. Similarly, $\sce_1[A_{a-1}^-]$ spans a submodule $\overline{L}(a)'(\cong \overline{L}(a))$ of $L(a)$. Moreover, (cf. \cite[Theorem 1]{mz})
 $$\overline{L}(a+1)\cong L(a+1)/\overline{L}(a),\qquad
 \overline{L}(a-1)\cong L(a)/\overline{L}(a)'.$$
Hence, $\overline{L}(a), 0\leq a\leq r,$ form a complete set of all irreducible $\sS_F$-modules. \end{proof}

\begin{remark} The classification of irreducible modules for $\sS_k(1|1,r)$ in the semisimple case is consistent with a classification given in \cite[Theorem 7.5]{DR}. 

\end{remark}

\begin{lemma}\label{ss2}With the same assumption on $l$ as in Lemma \ref{ss1},
the superalgebras $\sS_F(2|1, r)$ and $\sS_F(1|2, r)$ are not semisimple for all $r\geq l.$
\end{lemma}
\begin{proof} By Lemma \ref{DR5.8}, it suffices to consider $\sS_F=\sS_F(2|1,r)$.
Let $e=1_{(r,0,0)}$. Then, for $P=\sS_Fe$, End$_{\sS_F}(P)\cong F$ and so $P$ is an indecomposable $\sS_F$-module. We now show the existence of a proper submodule of $P$ if $r\geq l$.
Observe that $P$ is spanned by all $[A]$ with $\co(A)=(r,0,0)$. Such $A$ will be written as $A_{a,b,c}$ where $(a,b,c)^t$ is the first column of $A$. We have two cases to consider.

{\bf Case 1.} If $r=al+b$ with $0\leq b\leq l-2$ (i.e., $l\nmid r+1$), then $b+1<l$ and 
$\scf_1^{(b+1)}e=[A_{al-1,b+1,0}]\in P$. We now claim that $[A_{al-1,b+1,0}]$ is a maximal vector in 
the sense that $\sce_h^{(p)}[A_{al-1,b+1,0}]=0$ for all $h=1,2$ and $p\geq1$. This is clear if $h=2$ since all $a_{h+1,k}=a_{3,k}=0$. Also, by Theorem \ref{KeyMF}(1), we have $\sce_1^{(p)}[A_{al-1,b+1,0}]=0$ for $p>b+1$ and, for $p\leq b+1<l$, 
$$\sce_1^{(p)}[A_{al-1,b+1,0}]=\frac{\sce_1^{p-1}}{[p]_\ups^!}\sce_1[A_{al-1,b+1,0}]=\frac{\ups^{al-1}[\![al]\!]_{q^{-1}}}{[p]_\ups^!}\sce_1^{p-1}[A_{al,b,0}]=0.$$
By the claim, we see that $P':=\sS_F[A_{al-1,b+1,0}]=\sS_F^-[A_{al-1,b+1,0}]$ is a proper submodule of $P$ since $e\not\in P'$.

{\bf Case 2.} If $r=al-1$ (and so $a\geq2$), then by Theorem \ref{KeyMF},
$\scf_2(\scf_1^{(l)}e)=\scf_2[A_{r-l,l,0}]=[A_{r-l,l-1,1}]\in P.$
Now, since $r-l+1=(a-1)l$, we have $\sce_1[A_{r-l,l-1,1}]=\ups^{r-l}[\![r-l+1]\!]_{q^{-1}}[A_{r-l+1,l-2,1}]=0$ and
$\sce_2[A_{r-l,l-1,1}]=\ups^{l-1}[\![l]\!]_{q^{-1}}[A_{r-l,l,0}]=0$. Hence, $\sce_h^{(p)}[A_{r-l,l-1,1}]=0$ for all $h=1,2$ and $p<l$. Similarly, by Theorem \ref{KeyMF}(1),
$\sce_h^{(p)}[A_{r-l,l-1,1}]=0$ for $h=1,2$ and $p\geq l$. This proves that $\sS_F[A_{r-l,l-1,1}]=\sS^-_F[A_{r-l,l-1,1}]$ is a proper submodule of $P$. 

Combining the two cases, we conclude that 
$\sS_F$ is not semisimple whenever $r\geq l$.
\end{proof}
The following result is the quantum analogue of a result of F. Marko and A.N. Zubkov \cite{mz2}, which is stated in the abstract.
\begin{theorem}\label{thmqss}Let $F$ be a field containing elements $q\neq0,1$ and $\ups=\sqrt{q}$.
Then the  $q$-Schur superalgebra   $\sS_F(m|n, r)$ with $m,n\geq1$ is semisimple if and
only if one of the following holds:
\begin{itemize}
\item[(1)] $q$ is not a root of unity;

\item[(2)] $q$ is  a primitive $l$th root of unity and $r<l;$

\item[(3)] $m=n=1$ and $q$ is an $l$th root of unity with $l\nmid r.$
\end{itemize}
\end{theorem}
\begin{proof} The first two condition implies that $\sH_F$ is semisimple and so is $\sS_F$. The semisimplicity under (3) follows from Lemma \ref{ss1}. We now show that, if all three conditions fail, then $\sS_F$ is not semisimple. By Lemmas \ref{DR5.8}\&\ref{ss1}, it is suffices to look at the case for $m\geq2$ and $n\geq1$ and $l\leq r$. 

Consider the subset 
$$\La(m|n,r)' =\{\la\in\La(m|n,r)\mid \la^{(0)}=(\la_1,\la_2,0,\ldots,0),\la^{(1)}=(\la_{m+1},0,\ldots,0)\}$$ and let $f=\sum_{\la\in \La(m|n,r)' }1_\la$ and $e=1_{(r,0,\ldots,0)}$. Then $ef=e=fe$ and  it is clear that
there is an algebra isomorphism $\sS_F(2|1,r)\cong f\sS_F(m|n,r)f$. By identifying the two algebras under this isomorphism, we see that there is an $f\sS_F(m|n,r)f$-module isomorphism $\sS_F(2|1,r)1_{(r,0,0)}\cong f\sS_F(m|n,r)e$. This $f\sS_F(m|n,r)f$-module is indecomposable, but not irreducible, by Lemma \ref{ss2}.
  Since $\sS_F(m|n,r)e$ is indecomposable and its image
$f\sS_F(m|n,r)e$ under the ``Schur functor'' is indecomposable, but not irreducible, we conclude that $\sS_F(m|n,r)e$ is not irreducible (see \cite[(6.2g)]{Gr}). Hence, $\sS_F(m|n,r)$ is not semisimple.
\end{proof}

\begin{remark}Semisimple $q$-Schur algebras have been classified by K. Erdmann and D. Nakano
\cite[Theorem(A)]{EN}. By Corollary \ref{KMFcor}, we may also use this new approach to get their result; see Appendix A.
\end{remark}

\section{Infinitesimal and little $q$-Schur superalgebras}

We now give another application of the multiplication formulas. We first construct certain subsuperalgebras of the $q$-Schur superalgebra $\sS_R(m|n,r)$ over the commutative ring $R$ in which $q=\ups^2\neq1$ is a primitive $l$-th root of unity. (So $l\geq2$.) 

Let $\fks_R(m|n,r)$ be the $R$-submodule spanned by all $[A]$ with $A\in M(m|n,r)_l$, where
$$M(m|n,r)_l=\{(a_{i,j})\in M(m|n,r)\mid a_{i,j}<l\;\forall i\neq j\}.$$
We have the following super analogue of the infinitesimal $q$-Schur algebras (cf. \cite{CGW}).

\begin{theorem}\label{6.1} The $R$-submodule $\mathfrak{s}_R(m|n,r)$ is a subsuperalgebra generated by $\sce_h,\scf_h,1_\la$ for all $1\leq h<m+n$, $\la\in\La(m|n,r)$.
\end{theorem}

\begin{proof}Let $\fks'_R(m|n,r)$ be the subalgebra generated by $[aE_{h,h+1}+D]$ and $[bE_{h+1,h}+D']$, where $D,D'$ are diagonal matrices with  $aE_{h,h+1}+D,bE_{h+1,h}+D'\in M(m|n,r)_l$ and $0\leq a,b<l$.
Observe from the multiplication formulas in Theorem \ref{KeyMF} that if $A\in M(m|n,r)_l$ then  $\sce_h^{(a)}[A]=[aE_{h,h+1}+D][A]$ and $\scf_h^{(b)}[A]=[bE_{h+1,h}+D'][A]$, for some $D,D'$,  are linear combinations of $[B]$ with $B\in M(m|n,r)_l$. This implies that
 $\fks'_R(m|n,r)\subseteq \fks_R(m|n,r)$. Now, by the triangular relation \cite[Theorem 7.4]{DG}:
 \begin{equation}\label{tri}
 \prod_{i\leq h<j}^{(\leq_2)}[a_{j,i}E_{h+1,h}+D_{i,h,j}]\prod_{i\leq h<j}^{(\leq_1)}[a_{i,j}E_{h,h+1}+D_{i,h,j}]=(-1)^{\ol{A}}[A]+\text{lower terms},
 \end{equation}
  an inductive argument on the Bruhat order on $M(m|n,r)$ shows that every $[A]$ with $A\in M(m|n,r)_l$ belongs to $\fks'_R(m|n,r)$. Hence, 
 $\fks_R(m|n,r)= \fks'_R(m|n,r)$ is a subalgebra and, hence, a subsuperalgebra. From the argument above, we see easily that $\sce_h,\scf_h,1_\la$ can be generators.
\end{proof}


\begin{remarks} By \cite[Corollary 8.4]{DGW}, $\fks_R(m|n,r)$  is isomorphic to the infinitesimal $q$-Schur superalgebra defined
in \cite[\S3]{CGW} by using quantum coordinate superalgebra. 

\end{remarks}

We now construct a subsuperalgebra $\fku_R(m|n, r)$.
Let $\mathbb Z_l:=\mathbb Z/l\mathbb Z$ and let $\bar\;:\mathbb Z\to \mathbb Z_l$ be the quotient map. Extend this map to $M(m|n,r)$, $\Lambda(m|n,r)$ by baring on the entries. Thus, we may identify the image $\overline{M(m|n,r)}$ with the following set:
$$\overline{M(m|n,r)}=\{A^\pm+\diag(\overline{\partial}_A)\mid A\in M(m|n,r)_l\}=\overline{M(m|n,r)_l}.$$
where $A^\pm$ is obtained by replacing the diagonal of $A$ with 0's and $\partial_A \in \mathbb Z^{m+n}$ is the diagonal of $A$ (i.e., $A=A^\pm+\diag(\partial_A)$). For $A=A^\pm+\diag(\overline{\partial}_A)\in \overline{M(m|n,r)}$, define
$$
\ol{\xi}_A=\sum_{\lambda\in\Lambda(m|n,r-|A^\pm|)\atop \ol\lambda=\overline{\partial}_A}[A^\pm+\diag(\lambda)]=
\sum_{\lambda\in\Lambda(m|n,r-|A^\pm|)\atop \ol\lambda=\overline{\partial}_A}\xi_{A^\pm+\diag(\lambda)},$$
and let $\ol{1}_\la=\ol{\xi}_{\diag(\la)}$. Note that every $\ol{\xi}_A$ is a homogeneous element with respect the super structure on $\sS_R(m|n,r)$. 

We now have the super analogue of the {\it little} $q$-Schur algebra introduced in \cite{DFW}.

\begin{corollary}\label{little}The subsuperspace $\fku_R(m|n,r)$ of $\fks_R(m|n,r)$ spanned by $\ol{\xi}_A$ for all $A\in \overline{M(m|n,r)}$ is a subsuperalgebra with identity $\sum_{x\in\ol{\La(m|n,r)}}\ol{1}_{\diag(x)}$ and generated by $\sce_h,\scf_h,\ol1_{\la}$ for all $1\leq h<m+n,\la \in\ol{\La(m|n,r)}$.
\end{corollary}
\begin{proof} In this case, with a proof similar to that for Theorem \ref{6.1}, we see that $\fku_R(m|n,r)$ is the subalgebra generated by $\ol{\xi}_{aE_{h,h+1}+D}$ and $\ol{\xi}_{bE_{h+1,h}+D'}$, where $D,D'$ are diagonal matrices with  $aE_{h,h+1}+D,bE_{h+1,h}+D'\in \ol{M(m|n,r)}$.  Note that by taking the sum of the triangular relations \eqref{tri} for every  $A^\pm+\diag(\lambda)$ with $\overline{\la}=\overline{\partial}_A$, we obtain the required triangular relation for $\ol{\xi}_A$'s
(cf. the proof of \cite[Theorem 8.1]{DG}).
The last assertion is clear as every $\ol{\xi}_{aE_{h,h+1}+D}$ or $\ol{\xi}_{bE_{h+1,h}+D'}$ has the form $\sce_h^{(a)}\bar1_\la$ or $\scf_h^{(b)}\bar1_\la$.
\end{proof}

We end the paper with the following semisimplicity criteria for the infinitesimal/little $q$-Schur superalgebras;
compare the nonsuper case \cite[\S7]{DFW2} and \cite{Fu}.

\begin{theorem}\label{thmiqs} The superalgebra   $\fks_F(m|n, r)$ or $\fku_F(m|n, r)$ with $m,n\geq1$ is semisimple
if and only if one of the following holds:
\begin{itemize}
\item[(1)]  $r<l;$

\item[(2)] $m=n=1,l\nmid r.$
\end{itemize}
\end{theorem}
\begin{proof}We first look at the ``infinitesimal'' case. 
We observe that, if $r<l$ or $m=n=1$, then $\fks_F(m|n, r)=\sS_F(m|n,r)$. The ``if'' part is clear. Conversely, suppose $\fks_F(m|n, r)$ is semisimple. Since $\fks_F(1|1, r)=\sS_F(1|1,r)$, its semisimplicity forces $l\nmid r$. Assume $m\geq 2, n\geq1$ and $l\leq r$. By the proof of Lemma \ref{ss2},
we see that $\fks_F(2|1, r)e$ ($e=1_{(r,0,0)}$) is indecomposable and contains the proper submodule 
$\fks_F(2|1, r)[A_{al,b,0}]$ if $l\nmid r+1$, or  $\fks_F(2|1, r)[A_{r-l,l-1,1}]$ if $l\mid r+1$. Hence, we can use the Schur functor argument to conclude $\fks_F(m|n, r)$ is not semisimple unless $r<l$.

We now look at the ``little'' case. If $r<l$, then $\fku_F(m|n, r)=\sS_F(m|n,r)$ is semisimple. If $m=n=1$ and $l\nmid r$, then the simple module $L(a)$ constructed in the proof of Lemma \ref{ss1} remains irreducible when restricted to $\fku_F(m|n, r)$. This is seen from the last assertion of Corollary \ref{little}.
Thus, $\fks_F(m|n, r)$ as an $\fku_F(m|n, r)$-module is semisimple. As a $\fku_F(m|n, r)$-submodule of $\fks_F(m|n, r)$, $\fku_F(m|n, r)$ is semisimple. Conversely, if condition (1) and (2) both fail. Then $r\geq l$. If one of the $m$ and $n$ is great than 1, then $\fku_F(m|n,r)$ is not semisimple. To see this, it is enough to show that $M=\fks_F(2|1, r)e$ as an $\fku_F(2|1, r)$-module is indecomposable. Indeed, suppose $M=M_1\oplus M_2$ where $M_i$ are nonzero $\fku_F(2|1, r)$-submodules. Then, for any $\la\in\La(m|n,r)$,  $1_\la M_1$ and $1_\la M_2$ cannot be both non-zero since $\dim1_\la M=1$. This shows that $M_i$ is a direct sum of some $1_\la M$. Hence, $M_i$ is an $\fks_F(2|1, r)$-module, contrary to the fact that $M$ is an indecomposable $\fks_F(2|1,r)$-module. If $m=n=1$, then $l\mid r$. In this case, $\fku_F(1|1, r)$ is clearly non-semsimple as $\fku_F(1|1, r)\ol1_0$ is indecomposable, but not irreducible.
\end{proof}

\begin{appendix}
\section{A Theorem of Erdmann--Nakano}

\begin{theorem}[{\cite[Theorem(A)]{EN}}]\label{semqsch}
Let $F$ be a field of characteristic $p\geq0$ containing elements $q\neq0,1$ and
$\ups=\sqrt{q}$. Then the  $q$-Schur algebra   $\sS_F(m, r)$  is
semisimple if and only if one of the following holds:
\begin{itemize}
\item[(1)] $q$ is not a root of unity;

\item[(2)] $q$ is  a primitive $l$th root of unity and $r<l;$

\item[(3)] $m=2, p=0, l=2$ and $r$ is odd;
\item[(4)] $m=2, p\geq 3, l=2$ and $r$ is odd with $r<2p+1.$
\end{itemize}
\end{theorem}
\begin{proof}If $q$ satisfies (1) or (2), then $\sS_F(m, r)$ is clearly semisimple. Suppose now that $q$ is  a primitive $l$th root of unity and $r\geq l>1$. By Corollary \ref{KMFcor},
an argument similar to those given in the proofs of Lemma \ref{ss2} and Theorem \ref{thmqss} shows that both
 $\sS_F(m,r)1_{(r,0,\cdots,0)}, m\geq 3$, and
 $\sS_F(2,r)1_{(r,0)}, l\nmid r+1$,  are indecomposable but not irreducible.
In particular,
both $\sS_F(2,l)$ and $\sS_F(2,l+1)$ are not semisimple if $l\geq3$.
Since tensoring an $\sS_F(2,r)$-module with the determinant representation gives an
$\sS_F(2,r+2)$-module, we see that $\sS_F(2,r)$ is not
semisimple for all $r\geq l\geq 3$. Hence, a semisimple $\sS_F(m,r)$ forces $m=2,l=2$ and $2|r+1.$
It remains to determine the
semisimplicity of $\sS_F(2,r)$ when $r\geq l=2$ and $r$ odd (and so
$2|r+1$). We claim that, for $r\geq l=2$ with $r$ odd, $\sS_F(2,r)$ is semisimple if and only if either $p=0$ or $p\geq 3$  but $r<2p+1$. Indeed, $\sS_F(2,r)$ is semisimple if and only if  all $q$-Weyl modules $\Delta(\lambda), \lambda\in\Lambda^+(2,r)$, are irreducible.
For $\lambda=(\lambda_1,\lambda_2) \in\Lambda^+(2,r),$ if $x_\lambda\in \Delta(\lambda)$ is a highest weight vector, then
$\Delta(\lambda)$ has a basis
$
x_\lambda, \scf_1 x_\lambda, \scf_1^{(2)}x_\lambda,\cdots,
\scf_1^{(\lambda_1-\lambda_2)}x_\lambda
$ and, for $1\leq a\leq \lambda_1-\lambda_2,$ we have
\begin{equation*}\label{ss22}
\sce_1^{(a)}\scf_1^{(a)}x_\lambda =\sum_{s=0}^a \scf_1^{(a-s)}
{\left[ \lambda_{1}-\lambda_{2}; 2s-2a \atop s
\right]}_\ups \sce_1^{(a-s)}x_\lambda = {\left[
\lambda_{1}-\lambda_{2} \atop a \right]}_\ups x_\lambda.
\end{equation*}
Thus, the irreducibility of $\Delta(\la)$ is equivalent to
$
\prod_{0\leq a\leq
\lambda_{1}-\lambda_{2}} {\left[ \lambda_{1}-\lambda_{2}
\atop a \right]}_\ups \neq 0.
$
Since $r=\lambda_{1}+\lambda_{2}$ is odd and $l=2$, we see that $\lambda_{1}-\lambda_{2}$ is also odd and $
{\left[ \lambda_{1}-\lambda_{2} \atop a
\right]}_\ups= {\left( \frac{\lambda_{1}-\lambda_{2}-1}2 \atop a_1
\right)} {\left[ 1 \atop  a_0 \right]}_\ups, $ where
$a=2a_1+a_0$ with $a_0=0,1.$  Obviously, ${\left[ 1 \atop
a_0 \right]}_\ups=1.$  Thus, if $p=0$ or $p\geq 3$  but  $r<2p+1$  then ${\left(
\frac{\lambda_{1}-\lambda_{2}-1}2 \atop a_1 \right)}\neq 0$ for all $(\la_1,\la_2) \in\Lambda^+(2,r)$ and $1\leq a\leq \lambda_1-\lambda_2$. Hence,
$\sS_F(2,r)$ is semisimple in this case. Conversely, if $r\geq 2p+1,$
choose $\la$ so that $\lambda_{1}-\lambda_{2}=2p+1$  and $a=3$.  Then $ {\left[\lambda_{1}-\lambda_{2} \atop 3
\right]}_\ups={\left(
\frac{\lambda_{1}-\lambda_{2}-1}2 \atop 1 \right)}={\left( p \atop 1
\right)}=0.$ Hence, $\Delta(\lambda)$ is not simple in this case and so $\sS_F(2,r)$ is not semisimple.
\end{proof}

 \end{appendix}

\end{document}